\documentclass[12pt,a4paper]{article}
\usepackage{oldgerm}
\usepackage{amssymb}
\usepackage{amsthm}
\newtheorem{thm}{Theorem}[section]
\newtheorem{prop}[thm]{Proposition}
\newtheorem{lem}[thm]{Lemma}

\newtheorem{rem}{Remark}[section]
\newcommand{\Hcapg}{H_{\rho}}
\newcommand{\Hcapgc}{H_{\rho}^\mathbb{C}}
\newcommand{\hcapg}{\mathbf{h}_{\rho}}
\newcommand{\hcapgc}{\mathbf{h}_{\rho}^\mathbb{C}}
\newcommand{\hK}{hyper-K\"ahler\ }

\newcommand{\ImH}{{\rm Im}\mathbb{H}}
\begin{document}
\title{A generalization of Taub-NUT deformations}
\author{Kota Hattori}
\date{}
\maketitle
{\begin{center}
{\it Graduate School of Mathematical Sciences, University of Tokyo\\
3-8-1 Komaba, Meguro, Tokyo 153-8914, Japan\\
khattori@ms.u-tokyo.ac.jp}
\end{center}}
\maketitle
{\abstract We introduce a generalization of Taub-NUT 
deformations for 
large families of \hK quotients including toric \hK manifolds and 
quiver varieties. 
It is well-known that Taub-NUT deformations are defined for 
toric \hK manifolds, and the similar deformations were introduced for 
ALE \hK manifolds of type $D_k$ by Dancer, 
using the complete \hK metric on the cotangent bundle of 
complexification of compact Lie group.
We generalize them and study the Taub-NUT deformations
for the Hilbert schemes of $k$ points on $\mathbb{C}^2$.}
\section{Introduction}
\subsection{Taub-NUT spaces}
A hyper-K\"ahler manifold is a Riemannian manifold $(M,g)$ equipped with 
orthogonal integrable complex structures $I_1,I_2,I_3$ 
with quaternionic relations
$I_1^2 = I_2^2 = I_3^2 = I_1I_2I_3 = -1$ so that 
each $(M, g, I_i)$ is K\"ahlerian.
Then $M$ admits three symplectic forms $\omega_i := g(I_i\cdot,\cdot)$,
each of which is Ricci-flat K\"ahler metric with respect to $I_i$.
Throughout of this article, we regard $(M,I_1)$ as a complex manifold 
with a Ricci-flat K\"ahler metric 
$\omega_1$ and a non-degenerate closed $(2,0)$-form 
$\omega_2 + \sqrt{-1}\omega_3$, so-called a holomorphic symplectic structure.

The Euclidean space $\mathbb{C}^2 = \mathbb{R}^4$ is the trivial example of 
complete hyper-K\"ahler manifold, whose Ricci-flat K\"ahler metric is Euclidean
and the holomorphic symplectic structure is 
given by $dz\wedge dw$, where $(z,w)\in \mathbb{C}^2$ is the 
standard holomorphic coordinate.

In \cite{H}, Hawking constructed a complete hyper-K\"ahler metric on 
$\mathbb{R}^4$ with cubic volume growth which is called a Taub-NUT space.
On the other hand, LeBrun \cite{L} showed that the Taub-NUT space and 
the Euclidean space $\mathbb{C}^2$
are isomorphic as holomorphic symplectic manifolds,
consequently biholomorphic.
It means that the complex manifold $\mathbb{C}^2$ admits at least 
$2$ complete Ricci-flat K\"ahler metrics which are not isometric.
Such a phenomenon should never occur on compact complex manifolds, due to 
the uniqueness of the Ricci-flat K\"ahler metrics in each K\"ahler class.
The similar relation also holds between multi Eguchi-Hanson spaces and 
multi Taub-NUT spaces.

A generalization to the higher dimensional case are obtained by 
Gibbons, Rychenkova, Goto \cite{GRG} and Bielawski \cite{B}.
They construct Taub-NUT like \hK metrics by deforming the toric \hK manifolds,
using tri-Hamiltonian torus actions and \hK quotient method.

In \cite{D}, Dancer has defined the analogy of Taub-NUT deformations 
for some of 
the ALE spaces of type $D_k$ using $U(2)$-actions.
His results are based on the existence of \hK metrics on $T^*G^{\mathbb{C}}$ 
for any compact Lie group $G$ constructed by Kronheimer \cite{K2}.
Another generalization to noncommutative case is considered in Section 5 
of \cite{DS2}. 
They considered {\it hyper-K\"ahler modifications} for 
\hK manifolds with a tri-Hamiltonian $H$-action, where $H$ is a compact Lie group 
which is possibly noncommutative.
Note that the case of \cite{DS2} does not contains the results in \cite{D}, 
since the ALE spaces of type $D_k$ have no nontrivial tri-Hamiltonian actions.

In this paper, we generalize Taub-NUT deformations for some kinds of 
\hK quotients, 
which enable us to treat the above three cases \cite{GRG,B}, \cite{D} 
and Section 5 of \cite{DS} uniformly.
As a consequence, we apply the Taub-NUT deformations for 
the Hilbert schemes of $k$-points on $\mathbb{C}^2$.

\subsection{Notation and a main result}\label{sec1.2}
Here, we describe the main result in this paper more precisely.
Let a compact connected Lie group $H$ act on a \hK manifold $(M,g,I_1,I_2,I_3)$
preserving the \hK structure and there exists a \hK moment map
$\hat{\mu} : M\to \ImH\otimes \mathbf{h}^*$ with respect to $H$-action,
where $\ImH \cong\mathbb{R}^3$ be the pure imaginary part of quaternion 
$\mathbb{H}$ and $\mathbf{h}^*$ is the dual space of the Lie algebra 
$\mathbf{h} = {\rm Lie} (H)$.
Moreover, suppose the $H$-action extends to holomorphic 
$H^{\mathbb{C}}$ action on $(M,I_1)$,
where $H^{\mathbb{C}}$ is the complexification of $H$.
Let $\rho:H\to G\times G$ is a homomorphism of Lie groups, where $G$ is compact
connected Lie group.
Then $H_{\rho}:=\rho^{-1}(\Delta_G) \subset H$ acts on $M$, where
$\Delta_G\subset G\times G$ is the diagonal subgroup,
and the inclusion $\iota:H_{\rho}\to H$ induces a \hK moment map
$\mu := \iota^*\circ \hat{\mu}$.
If we denote by $Z_H\subset \mathbf{h}^*$ the subspace of fixed points by 
coadjoint action of $H$ on $\mathbf{h}^*$, then we have the \hK quotient 
$\mu^{-1}(\iota^*\zeta)/H_{\rho}$ for each $\zeta \in \ImH \otimes Z_H$.
In this paper we define Taub-NUT deformations for 
$\mu^{-1}(\iota^*\zeta)/H_{\rho}$ by the following way.

Let $N_G=T^*G^{\mathbb{C}}$ be the \hK manifolds with a $G\times G$-action 
constructed by Kronheimer \cite{K2}, and
$\nu:N_G \to \ImH\otimes (\mathbf{g}\oplus \mathbf{g})^*$ be its \hK
moment map described by Dancer and Swann \cite{DS}.
Then $H$ acts on $M\times N_G$ by $\rho$, and for $(x,p)\in M\times N_G$, 
$\sigma(x,p):= \hat{\mu} (x) + 
\rho^*(\nu(p))$ becomes the \hK moment map,
accordingly we obtain a \hK quotient $\sigma^{-1}(\zeta)/H$ for each 
$\zeta \in \ImH\otimes Z_H$.
Now we have two \hK quotients $\mu^{-1}(\iota^*\zeta)/H_{\rho}$ and 
$\sigma^{-1}(\zeta)/H$.
If they are smooth, then there are Ricci-flat K\"ahler metrics 
$\omega^{\iota^*\zeta}_{1},\omega^{\zeta}_{1}$ and 
holomorphic symplectic structures 
$\omega^{\iota^*\zeta}_{2} + \sqrt{-1}\omega^{\iota^*\zeta}_{3}$, 
$\omega^{\zeta}_{2} + \sqrt{-1}\omega^{\zeta}_{3}$, respectively.
Next we extend $\rho$ to the holomorphic homomorphism $H^{\mathbb{C}}\to 
G^{\mathbb{C}}\times G^{\mathbb{C}}$ and obtain a holomorphic map
\begin{eqnarray}
\bar{\rho} : H^{\mathbb{C}} / \Hcapgc \to 
(G^{\mathbb{C}}\times G^{\mathbb{C}})/ \Delta_{G^{\mathbb{C}}}.\label{rho1}
\end{eqnarray}
Then the main result is described as follows.

\begin{thm}
Let $M = \mathbb{H}^N$, and
$H\subset Sp(N)$ acts on $M$ naturally. 
Assume $\bar{\rho}$ is surjective.
Then there exists an biholomorphism as complex analytic spaces
\begin{eqnarray}
\psi:\mu^{-1}(\iota^*\zeta)/H_{\rho} 
\to \sigma^{-1}(\zeta)/H\nonumber
\end{eqnarray}
for each $\zeta\in Z_H$.
Moreover, if $\Hcapg$ acts on $\mu^{-1}(\iota^*\zeta)$ freely, 
then $\mu^{-1}(\iota^*\zeta)/H_{\rho}$ and $\sigma^{-1}(\zeta)/H$ are 
smooth complete \hK manifolds and $\psi$ satisfies
\begin{eqnarray}
[\psi^*\omega^{\zeta}_{1}]_{DR}
= [\omega^{\iota^*\zeta}_{1}]_{DR}, \quad
\psi^*(\omega^{\zeta}_{2} + \sqrt{-1}\omega^{\zeta}_{3})
= \omega^{\iota^*\zeta}_{2} + \sqrt{-1}\omega^{\iota^*\zeta}_{3},
\nonumber
\end{eqnarray}
where $[\cdot]_{DR}$ is a de Rham cohomology class.
\label{1.1}
\end{thm}

\subsection{Hilbert schemes of $k$ points on $\mathbb{C}^2$}
We apply the above theorem to $M = {\rm End}(\mathbb{C}^k)\otimes_{\mathbb{C}} 
\mathbb{H}$, $H=U(k) \times U(k)$, $G= U(k)$ and $\rho = {\rm id}$.
Then $Z_H = \mathbb{R}$, and $\mu^{-1}(\iota^*\zeta)/H_{\rho}$ becomes
a quiver varieties constructed in \cite{Na1}, called the Hilbert scheme of 
$k$-points of $\mathbb{C}^2$.
In particular, $\mu^{-1}(0)/H_{\rho}$ is isomorphic to 
$(\mathbb{C}^2)^k/\mathcal{S}_k$
with Euclidean metric as \hK orbifolds.
In this case, 
$\sigma^{-1}(\zeta)/H$ becomes a smooth \hK manifolds diffeomorphic to
$\mu^{-1}(\iota^*\zeta)/H_{\rho}$ by Theorem \ref{1.1},
and the \hK metric on $\sigma^{-1}(0)/H$ can be described concretely.

\begin{thm}
In the above situation, we have an isomorphism
\begin{eqnarray}
\sigma^{-1}(0)/H\cong (\mathbb{C}_{Taub-NUT}^2)^k/\mathcal{S}_k\nonumber
\end{eqnarray}
as \hK orbifolds,
where $\mathbb{C}_{Taub-NUT}^2$ is the Taub-NUT space.
\label{TNUTHilb}
\end{thm}

\subsection{Outline of the proof}
Theorem \ref{1.1} is proven in the following way.
The hyper-K\"ahler moment map $\mu$ is decomposed into
$\mu = (\mu_{1}, \mu_{\mathbb{C}}:= \mu_{2} + \sqrt{-1} \mu_{3})$
along the decomposition ${\rm Im}\mathbb{H} = \mathbb{R} \oplus \mathbb{C}$,
where $\mu_{1}:M\to \hcapg^*$ and $\mu_{\mathbb{C}}: M \to (\hcapgc)^*$,
and the other hyper-K\"ahler moment maps and the parameter
$\zeta \in {\rm Im}\mathbb{H}\otimes Z_{H}$ are also decomposed 
in the same manner.
Define sets of ``stable points" by
\begin{eqnarray}
\mu_{\mathbb{C}}^{-1}(\iota^*\zeta_{\mathbb{C}})_{\iota^*\zeta_1}
&:=& \Hcapgc\cdot \mu^{-1}(\iota^*\zeta)\nonumber\\
&=& \{g\cdot x \in \mu_{\mathbb{C}}^{-1}(\iota^*\zeta_{\mathbb{C}});\ 
g\in \Hcapgc,\ x\in \mu^{-1}(\iota^*\zeta)\}
,\nonumber\\
\sigma_\mathbb{C}^{-1}(\zeta_{\mathbb{C}})_{\zeta_1}
&:=& H^{\mathbb{C}} \cdot \sigma^{-1}(\zeta).\nonumber
\end{eqnarray}
The natural embedding 
\begin{eqnarray}
\mu^{-1}(\iota^*\zeta)
\hookrightarrow \mu_{\mathbb{C}}^{-1}(\iota^*
\zeta_{\mathbb{C}})_{\iota^*\zeta_1},
\quad 
\sigma^{-1}(\zeta)\hookrightarrow 
\sigma_\mathbb{C}^{-1}(\zeta_{\mathbb{C}})_{\zeta_1}\nonumber
\end{eqnarray}
induce 
\begin{eqnarray}
\mu^{-1}(\iota^*\zeta)/\Hcapg \to 
\mu_{\mathbb{C}}^{-1}(\iota^*\zeta_{\mathbb{C}})_{\iota^*\zeta_1} / \Hcapgc,
\quad 
\sigma^{-1}(\zeta)/H \to
\sigma_\mathbb{C}^{-1}(\zeta_{\mathbb{C}})_{\zeta_1} / H^{\mathbb{C}},\nonumber
\end{eqnarray}
which are isomorphisms of complex analytic spaces 
by \cite{HL}.
Here, to regard 
$\mu_{\mathbb{C}}^{-1}(\iota^*\zeta_{\mathbb{C}})_{\iota^*\zeta_1} / \Hcapgc$ 
and
$\sigma_\mathbb{C}^{-1}(\zeta_{\mathbb{C}})_{\zeta_1} / H^{\mathbb{C}}$ 
as complex analytic spaces, 
we consider the sets of ``semistable points" 
$\mu_{\mathbb{C}}^{-1}(\iota^* \zeta_{\mathbb{C}})_{\iota^*\zeta_1}^{ss}$
and $\sigma_\mathbb{C}^{-1}(\zeta_{\mathbb{C}})_{\zeta_1}^{ss}$ 
in Section \ref{pf.of1.1}.

Thus the proof of Theorem \ref{1.1} is reduced to 
construct an isomorphism between
$\mu_{\mathbb{C}}^{-1}
(\iota^*\zeta_{\mathbb{C}})_{\iota^*\zeta_1} / \Hcapgc$ and
$\sigma_\mathbb{C}^{-1}(\zeta_{\mathbb{C}})_{\zeta_1} 
/ H^{\mathbb{C}}$.
First of all, we define an $\Hcapgc$ equivariant 
holomorphic map 
$\hat{\psi}: \mu_{\mathbb{C}}^{-1} (\iota^*\zeta_{\mathbb{C}}) \to 
\sigma_\mathbb{C}^{-1}(\zeta_{\mathbb{C}})$
so that a induced map $\psi: \mu_{\mathbb{C}}^{-1}
(\iota^*\zeta_{\mathbb{C}}) / \Hcapgc \to
\sigma_\mathbb{C}^{-1}(\zeta_{\mathbb{C}}) 
/ H^{\mathbb{C}}$ is bijective.

Then it suffices to see that $\psi$ gives a one-to-one correspondence 
between $\mu_{\mathbb{C}}^{-1}(\iota^*\zeta_{\mathbb{C}})_{\iota^*\zeta_1} 
/ \Hcapgc$ 
and $\sigma_\mathbb{C}^{-1}(\zeta_{\mathbb{C}})_{\zeta_1} /H^{\mathbb{C}}$.
To show it, we describe some equivalent conditions for 
$x$ and $\hat{\psi}(x)$ to be $x\in \mu_{\mathbb{C}}^{-1}
(\iota^*\zeta_{\mathbb{C}})_{\iota^*\zeta_1}$ and
$\hat{\psi}(x) \in \sigma_\mathbb{C}^{-1}(\zeta_{\mathbb{C}})_{\zeta_1}$
in Section \ref{hamilton}, 
using some convex functions on $G\backslash G^{\mathbb{C}}$
We also need the description of the K\"ahler potential of $N_G$, 
which is discussed in Section \ref{T*G^c}.

This paper is organized as follows.
We review the construction of \hK structures on $N_G = T^*G^{\mathbb{C}}$
along \cite{K2}, and describe \hK moment map by \cite{D} in Section \ref{T*G^c}.
Moreover we describe the K\"ahler potentials of the \hK metrics 
using the method of \cite{HKLR} and \cite{Do}.

In Section \ref{hamilton}, to obtain other description of $\mu_{\mathbb{C}}^{-1}
(\iota^*\zeta_{\mathbb{C}})_{\iota^*\zeta_1}$ and 
$\sigma_\mathbb{C}^{-1}(\zeta_{\mathbb{C}})_{\zeta_1}$, 
we study the relation between a K\"ahler moment map and 
some geodesically convex functions on Riemannian symmetric spaces.

In Section \ref{mainresult}, we prove Theorem \ref{1.1}, 
by using the description of 
K\"ahler potentials obtained in Section \ref{T*G^c} and the methods in 
Section 3.

In Section 5 we apply Theorem \ref{1.1} to the Hilbert schemes of 
$k$ points on $\mathbb{C}^2$ and show Theorem \ref{TNUTHilb}.
Moreover, we see that Theorem \ref{1.1} can be applied to 
quiver varieties and toric \hK varieties.

\section{Hyper-K\"ahler structures on $T^*G^{\mathbb{C}}$}\label{T*G^c}
\subsection{Riemannian description}
Here we review briefly the hyper-K\"ahler quotient construction of $N_G$
along \cite{K2}, and describe hyper-K\"ahler moment map $\nu$ along 
\cite{D}\cite{DS}.

Let $G$ be a compact connected Lie group, 
and $\|\cdot \|$ is a norm on $\mathbf{g}$ induced by an ${\rm Ad}_G$-invariant
inner product.
Consider the following equations
\begin{eqnarray}
\frac{dT_i}{ds} + [T_0,T_i] + [T_j,T_k] = 0 \quad {\rm for\ }(i,j,k) = 
(1,2,3),(2,3,1),(3,1,2),\label{nahm}
\end{eqnarray}
for $T:=(T_0, T_1, T_2, T_3)\in C^1([0,1],\mathbf{g})\otimes \mathbb{H}$.
Put
\begin{eqnarray}
\mathcal{N}_G:= \{T\in C^1([0,1],\mathbf{g})\otimes \mathbb{H};\ 
 {\rm equations\ (\ref{nahm}) \ holds }\}.\nonumber
\end{eqnarray}
Then a gauge group $\mathcal{G} := C^2([0,1],G)$ 
acts on $\mathcal{N}_G$ by
\begin{eqnarray}
g\cdot T:=({\rm Ad}_gT_0 + g\frac{d}{ds}g^{-1},{\rm Ad}_gT_1, 
{\rm Ad}_gT_2, {\rm Ad}_gT_3),\nonumber
\end{eqnarray}
and we obtain
\begin{eqnarray}
N_G := \mathcal{N}_G/\mathcal{G}_0,\nonumber
\end{eqnarray}
where $\mathcal{G}_0 := \{g\in \mathcal{G};\ g(0) = g(1) = 1\}$.
It is shown in \cite{K2} that $N_G$ becomes a $C^{\infty}$ manifold 
of dimesion $4\dim G$, and the standard hyper-K\"ahler structure on 
$C^1([0,1],\mathbf{g})\otimes \mathbb{H}$ induces a
hyper-K\"ahler structure $g_G,I_{G,1},I_{G,2},I_{G,3}$ on $N_G$.
Here, $g_G$ is induced from the $L^2$-inner product on
$C^1([0,1],\mathbf{g})\otimes \mathbb{H}$ using 
${\rm Ad}_G$-invariant inner product on $\mathbf{g}$.

Now we have a Lie group isomorphism 
$\mathcal{G}/\mathcal{G}_0 = G\times G$
defined by $g\mapsto (g(0),g(1))$.
Since $\mathcal{G}$ acts on $\mathcal{N}_G$, there exists a
$G\times G$-action on $N_G$ preserving the hyper-K\"ahler structure.

\begin{thm}[\cite{D}\cite{DS}]
The hyper-K\"ahler moment map 
$\nu = (\nu^0,\nu^1): N_G \to 
{\rm Im}\mathbb{H}\otimes (\mathbf{g}^*\oplus\mathbf{g}^*)$
with respect to the action of $G\times G$ on $N_G$
is given by
\begin{eqnarray}
\nu^0 ([T]) = (T_1(0),T_2(0),T_3(0)),\quad 
\nu^1 ([T]) = - (T_1(1),T_2(1),T_3(1)),\nonumber
\end{eqnarray}
under the identification $\mathbf{g}^* \cong \mathbf{g}$
using ${\rm Ad}_G$-invariant inner product.
Here we denote by $[T]\in N_G$ the equivalence class 
represented by $T\in \mathcal{N}_G$.
\label{Dancer}
\end{thm}

\subsection{ Holomorphic description}\label{3.2}
In this subsection, we review that $(N_G, I_{G,1})$ is identified
with a holomorphic cotangent bundle 
$T^*G^{\mathbb{C}} \cong G^{\mathbb{C}} \times \mathbf{g}^{\mathbb{C}}$
as complex manifolds along \cite{Do}\cite{K2}.

For each $T\in \mathcal{N}_G$,
there exists a solution $u:[0,1]\to G^{\mathbb{C}}$
 for an ordinary differential equation
\begin{eqnarray}
\frac{du}{ds}u^{-1} = -(T_0 + \sqrt{-1}T_1),\nonumber
\end{eqnarray}
then $T_2 + \sqrt{-1}T_3 = {\rm Ad}_{u(s)u(0)^{-1}}\eta$ for some
$\eta\in \mathbf{g}^{\mathbb{C}}$ \cite{K2}.
Then a holomorphic map 
$\Phi : (N_G,I_{G,1}) \to G^{\mathbb{C}} \times \mathbf{g}^{\mathbb{C}}$
is obtained by $[T]\mapsto (u(0)u(1)^{-1},\eta)$.
\begin{thm}[\cite{Do}\cite{K2}]
The map $\Phi$ is biholomorphic and preserves holomorphic symplectic 
structures.
\label{K2Do}
\end{thm}
The moment map $\nu$ is decomposed into 
$\nu = (\nu_{\mathbb{R}}= \nu_1,\nu_{\mathbb{C}} = \nu_2 +\sqrt{-1}\nu_3)$
along the decomposition ${\rm Im}\mathbb{H}\cong \mathbb{R}\oplus \mathbb{C}$.
Then 
\begin{eqnarray}
\nu_{\mathbb{C}}(Q,\eta) = (\eta, - {\rm Ad}_{Q^{-1}} (\eta))\nonumber
\end{eqnarray}
under the identification $N_G = 
G^{\mathbb{C}} \times \mathbf{g}^{\mathbb{C}}$.
$\nu_{\mathbb{C}}$ is a holomorphic moment map
with respect to $G^{\mathbb{C}}\times G^{\mathbb{C}}$ action on 
$T^*G^{\mathbb{C}}$.
This action is given by
\begin{eqnarray}
(g_0,g_1)(Q,\eta) = (g_0 Q g_1^{-1},{\rm Ad}_{g_0}\eta) \label{groupaction}
\end{eqnarray}
for $(g_0,g_1)\in G^{\mathbb{C}}\times G^{\mathbb{C}}$ and 
$(Q,\eta) \in G^{\mathbb{C}} \times \mathbf{g}^{\mathbb{C}}$.

\subsection{K\"ahler potentials}
Next we describe the K\"ahler potential of 
the K\"ahler manifold $(N_G, g_{G}, I_{G,1})$.
We apply the following results for $N_G$.
\begin{lem}[\cite{HKLR}]
Let $(M,g,I_1,I_2,I_3)$ be a hyper-K\"ahler manifold with
isometric $S^1$ action generated by a Killing field $X$, 
which satisfies
\begin{eqnarray}
L_X\omega_1 = \omega_2,\quad L_X\omega_2 = - \omega_1,
\quad L_X\omega_3 = 0,\nonumber
\end{eqnarray}
where $\omega_i$ is the K\"ahler form of $(M,g,I_i)$.
If $\mu$ is the moment map with respect to the
$S^1$-action on $(M,\omega_3)$, then 
$\omega_1 = 2\sqrt{-1}\partial_1\overline{\partial}_1\mu$,
$\omega_2 = 2\sqrt{-1}\partial_2\overline{\partial}_2\mu$.
\label{potential}
\end{lem}
We apply this lemma as follows.
Let $\omega_{G,i}:= g_G(I_{G,i}\cdot,\cdot)$, and define 
$e^{i\theta} \cdot [T]:=[T_0,\cos\theta\ T_1 + \sin\theta\ T_2, 
- \sin\theta\ T_1 + \cos\theta\ T_2, T_3]$
for $[T]\in N_G$, which is an $S^1$-action on $N_G$
preserving $\omega_{G,3}$ and satisfies the assumption of 
Lemma (\ref{potential}). 
Then the moment map is given by
$\|T_1\|^2_{L^2} + \|T_2\|^2_{L^2} 
= \int_0^1(\|T_1(s)\|^2 + \|T_2(s)\|^2)ds$.

Moreover, we have another $S^1$-action defined by
$e^{i\theta} \cdot [T]:=[T_0,-\sin\theta\ T_3 + \cos\theta\ T_1, 
T_2, \cos\theta\ T_3 + \sin\theta\ T_1]$,
which preserves $\omega_{G,2}$.
In this case the moment map is given by
$\|T_1\|^2_{L^2} + \|T_3\|^2_{L^2}$.
Thus we obtain the followings from Lemma \ref{potential}
\begin{prop}
Put $\mathcal{E} : N_G \to \mathbb{R}$ to be
\begin{eqnarray}
\mathcal{E}([T]) := \|T_1\|^2_{L^2} + \frac{1}{2}(\|T_2\|^2_{L^2} + \|T_3\|^2_{L^2}).\nonumber
\end{eqnarray}
Then $\omega_{G,1} =  2\sqrt{-1}\partial_1\overline{\partial}_1\mathcal{E}$.
\label{potential2}
\end{prop}

Next we describe the K\"ahler potential $\mathcal{E}$
as a function on $T^*G^{\mathbb{C}} =
 G^{\mathbb{C}} \times \mathbf{g}^{\mathbb{C}}$.
The ${\rm Ad}_G$-invariant inner product on $\mathbf{g}$ induces 
a homogeneous Riemannian metric on $G\backslash G^{\mathbb{C}}$.
Define an antiholomorphic involution of $G^{\mathbb{C}}$ 
by $(ge^{\sqrt{-1}\xi})^* := e^{\sqrt{-1}\xi} g^{-1}$ for $g\in G$
and $\xi \in \mathbf{g}$, by using the polar decomposition 
$G^{\mathbb{C}} \cong G\cdot \exp (\sqrt{-1}\mathbf{g})$.
Then $G\backslash G^{\mathbb{C}}$ is identified with 
$\exp (\sqrt{-1}\mathbf{g})$ by $G\cdot g\mapsto g^*g$ for 
$g\in G^{\mathbb{C}}$.
This metric on $G\backslash G^{\mathbb{C}}$ is naturally extended to 
a hermitian metric on $T_h (G\backslash G^{\mathbb{C}})\otimes \mathbb{C}$,
which is also denoted by $\|\cdot \|_h$.

\begin{prop}
For each $(Q,\eta)\in G^{\mathbb{C}} \times \mathbf{g}^{\mathbb{C}} 
= N_G$, 
\begin{eqnarray}
\mathcal{E}(Q,\eta) = \frac{1}{2}\min_{h\in P(a^*a,b^*b)}
\int_0^1 \bigg( \bigg\| \frac{dh}{ds} \bigg\|^2_{h} 
+ \|h{\rm Ad}_a^{-1}(\eta) \|_h^2\bigg) ds,\nonumber
\end{eqnarray}
where $ab^{-1} =Q$ and $P(h_0,h_1):=\{h\in C^{\infty}([0,1],
G\backslash G^{\mathbb{C}});\ h(0)=h_0, h(1)=h_1 \}$.
\label{kahlerpotential}
\end{prop}
\begin{proof}
The essential part of the proof is obtained in \cite{Do}, 
and we explain the outline.
Define $\mathcal{L} : C^1([0,1],\mathbf{g})
\otimes \mathbb{H} \to \mathbb{R}$ by
$\mathcal{L}(T):=\|T_1\|^2_{L^2} + \frac{1}{2}(\|T_2\|^2_{L^2} + 
\|T_3\|^2_{L^2})$.
Fix $T\in \mathcal{N}_G$,
then $\mathcal{L}|_{\mathcal{G}_0^{\mathbb{C}}\cdot T}$ 
attains its minimum value at $T$ by Lemma 2.3 of \cite{Do},
hence
\begin{eqnarray}
\mathcal{E}([T]) = \min_{g\in \mathcal{G}_0^{\mathbb{C}}} 
\mathcal{L}(g\cdot T).\nonumber
\end{eqnarray}
Here $\mathcal{G}_0^{\mathbb{C}}$ is the complexified gauge group
defined by
\begin{eqnarray}
\mathcal{G}_0^{\mathbb{C}}:=\{g\in \mathcal{G}^{\mathbb{C}} 
= C^2([0,1],G^{\mathbb{C}});\ 
g(0) = g(1) = 1\},\nonumber
\end{eqnarray}
and a $\mathcal{G}_0^{\mathbb{C}}$ action on 
$C^1([0,1],\mathbf{g}) \otimes \mathbb{H}$ is defined by
$g\cdot (T_0 +\sqrt{-1}T_1, T_2 +\sqrt{-1}T_3) := 
({\rm Ad}_g(T_0 +\sqrt{-1}T_1) + g\frac{d}{ds}g^{-1}, 
{\rm Ad}_g(T_2 +\sqrt{-1}T_3))$.
Take $u:[0,1]\to G^{\mathbb{C}}$ as in Section \ref{3.2},
and let $(Q,\eta) \in G^{\mathbb{C}} \times \mathbf{g}^{\mathbb{C}}$
be corresponding to $[T]\in N_G$ under the identification
given in Section \ref{3.2}.
Then we have $Q = u(0)u(1)^{-1}$, $T_2(s) +\sqrt{-1}T_3(s) 
= {\rm Ad}_{u(s)u(0)^{-1}}\eta$.
For $g\in \mathcal{G}_0^{\mathbb{C}}$,
\begin{eqnarray}
g\cdot T = \bigg(g(s)u(s)\frac{d}{ds}(gu)^{-1}(s), 
\ {\rm Ad}_{g(s)u(s)}{\rm Ad}_{u(0)^{-1}}\eta \bigg).\nonumber
\end{eqnarray}
Now we extend the ${\rm Ad}_G$-invariant inner product on $\mathbf{g}$ to the 
$\mathbb{C}$ bilinear form on $\mathbf{g}^{\mathbb{C}}$. 
From the calculation in the proof of Lemma 2.3 of \cite{Do},
we obtain
\begin{eqnarray}
\mathcal{L} (g\cdot T) = \int_0^1\big( \frac{1}{4}\| h' \|_h^2 
+ \frac{1}{2}\| h {\rm Ad}_{u(0)^{-1}}(\eta) \|_h^2 \big) ds,\nonumber
\end{eqnarray}
where $h = (gu)^*(gu)$.
Thus we obtain
\begin{eqnarray}
\mathcal{E}([T]) = \min_{h\in P(h_0,h_1)} 
\int_0^1\big( \frac{1}{4}\| h' \|_h^2 
+ \frac{1}{2}\| h {\rm Ad}_{u(0)^{-1}}(\eta) \|_h^2 \big) ds,\nonumber
\end{eqnarray}
where $h_0 = u(0)^*u(0)$, $h_1 = u(1)^*u(1)$ and $u(0)u(1)^{-1} = Q$,
hence we have the assertion.
\end{proof}

\section{K\"ahler manifolds with Hamiltonian actions}\label{hamilton}
Let $L^{\mathbb{C}}$ be a complexification of connected compact Lie group $L$,
and $L^{\mathbb{C}}$ acts on a complex manifold $(X,I)$ holomorphically.
Let $\omega$ be a K\"ahler form on $X$, and $L$ acts on $(X,I,\omega)$ 
isometrically.
Suppose that there is a symplectic moment map ${m}: X \to \mathbf{l}^*$ 
with respect to $L$ action, where $\mathbf{l}$ is the Lie algebra of $L$.

Fix an ${\rm Ad_L}$-invariant inner product on $\mathbf{l}$, then the Riemannian 
metric on the homogeneous space $L\backslash L^{\mathbb{C}}$ is induced.
We denote by $Lg\in L\backslash L^{\mathbb{C}}$ the equivalence class 
represented by $g\in L^{\mathbb{C}}$, and put 
\begin{eqnarray}
\xi_g:=\frac{d}{dt}\Big|_{t=0} Le^{\sqrt{-1}t\xi}g \in T_{Lg} 
(L\backslash L^{\mathbb{C}})\nonumber
\end{eqnarray}
for each $\xi \in \mathbf{l}$, where 
$\{Le^{\sqrt{-1}t\xi}g\}_{t\in \mathbb{R}}$ is a geodesic through $Lg$.

Now define a $1$-form $\alpha_{x,\zeta} \in 
\Omega^1(L\backslash L^{\mathbb{C}})$
by $(\alpha_{x,\zeta})_{Lg}(\xi_g):= \langle {m} (gx) - \zeta, \xi \rangle$
for $\zeta \in Z_L$, which is easily checked to be closed.
Since $L\backslash L^{\mathbb{C}} \cong \mathbf{l}$ is simply-connected,
$\alpha_{x,\zeta}$ is $d$-exact and there is a unique primitive function 
up to constant.
Accordingly, there is a unique function 
$\Phi_{x,\zeta}:L\backslash L^{\mathbb{C}} 
\to \mathbb{R}$ satisfying $d \Phi_{x,\zeta} = \alpha_{x,\zeta}$ and 
$\Phi_{x,\zeta}(L\cdot 1) = 0$. 
Here, the latter equality is a normalization for removing the ambiguity, 
and it is not essential. 
Then it is easy to check that $\Phi_{x,\zeta}$ is a geodesically convex 
function on the Riemannian symmetric space $L\backslash L^{\mathbb{C}}$.

From the argument in \cite{Ki}\cite{Ne}, the naturally induced map
\begin{eqnarray}
{m}^{-1}(\zeta)/L \to X_{\zeta} /L^{\mathbb{C}}\nonumber
\end{eqnarray}
becomes a homeomorphism, where
\begin{eqnarray}
X_{\zeta} := \{x\in X;\ \Phi_{x,\zeta}\ {\rm has\ a\ critical\ point} \}.
\nonumber
\end{eqnarray}
In this subsection we show some equivalent conditions for the existence of 
the critical point of $\Phi_{x,\zeta}$.

For $g\in L^{\mathbb{C}}$, we denote by $R_g$ the isometry on 
$L\backslash L^{\mathbb{C}}$ given by the right action of $L^{\mathbb{C}}$ 
on $L\backslash L^{\mathbb{C}}$.
Then we can check that $R_g^*\alpha_{x,\zeta} = \alpha_{gx,\zeta}$, 
and $R_g^*\Phi_{x,\zeta} - \Phi_{gx,\zeta}$ becomes a constant function, 
hence $X_{\zeta}$ is $L^{\mathbb{C}}$-closed.

Let
\begin{eqnarray}
{\rm Stab}(x)^{\mathbb{C}} := \{g\in L^{\mathbb{C}};\ gx = x\},\nonumber
\end{eqnarray}
and ${\rm stab}(x)^{\mathbb{C}}$ be 
the Lie algebra of ${\rm Stab}(x)^{\mathbb{C}}$.
We put
\begin{eqnarray}
{\rm Stab}(x) &:=& \{g\in L;\ gx = x\} = {\rm Stab}(x)^{\mathbb{C}}\cap L,
\nonumber\\
{\rm stab}(x) &:=& {\rm Lie}({\rm Stab}(x)) = {\rm stab}(x)^{\mathbb{C}}
\cap \mathbf{l}.\nonumber
\end{eqnarray}
Note that ${\rm Stab}(x)^{\mathbb{C}}$ contains the complexification of 
${\rm Stab}(x)$ as a subgroup, though it is not necessary to be equal.
Let $\pi_{\rm Im} : \mathbf{l}^{\mathbb{C}}\to \mathbf{l}$ be defined by
$\pi_{\rm Im}(a+\sqrt{-1}b)=b$ for $a,b\in\mathbf{l}$, and put
$\widetilde{{\rm stab}}(x):= 
\pi_{\rm Im}({\rm stab}(x)^{\mathbb{C}})$.
Then there is the orthogonal decomposition
$\mathbf{l} = \widetilde{{\rm stab}}(x) \oplus V_x$ with respect to 
${\rm Ad_L}$-invariant inner product on $\mathbf{l}$.

\begin{lem}
For each $g\in L^{\mathbb{C}}$, there are 
$\gamma \in {\rm Stab}(x)^{\mathbb{C}}$ and $\xi \in V_x$
such that $Lg\gamma = Le^{\sqrt{-1}\xi}$.
\label{reduce}
\end{lem}
\begin{proof}
Consider the smooth function $f:Lg\cdot{\rm Stab}(x)^{\mathbb{C}} 
\to \mathbb{R}$
defined by 
\begin{eqnarray}
f(Lg\gamma):= {\rm dist}_{L\backslash L^{\mathbb{C}}}
(L\cdot 1,Lg\gamma)^2,\nonumber
\end{eqnarray}
where $Lg\cdot{\rm Stab}(x)^{\mathbb{C}} \subset 
L\backslash L^{\mathbb{C}}$ is the 
${\rm Stab}(x)^{\mathbb{C}}$-orbit through $Lg$.
Now ${\rm Stab}(x)^{\mathbb{C}}$ is closed in $L^{\mathbb{C}}$, then 
$Lg\cdot{\rm Stab}(x)^{\mathbb{C}}$ is the closed orbit, hence $f$ is proper.
Since $f$ is bounded from below, there is a minimum point 
$Lg\gamma_0 \in Lg\cdot{\rm Stab}(x)^{\mathbb{C}}$.
By the polar decomposition $L^{\mathbb{C}}\cong L\times 
\mathbf{l}$, we can take $h\in L$ and 
$\xi \in \mathbf{l}$ such that $hg\gamma_0 = e^{\sqrt{-1}\xi}$.
Under the identification $T_{Lg\gamma_0}(L\backslash L^{\mathbb{C}})
\cong \mathbf{l}$, the subspace 
$T_{Lg\gamma_0} (Lg\cdot{\rm Stab}(x)^{\mathbb{C}})$ is identified 
with $\widetilde{{\rm stab}}(hg\gamma_0 x) = \widetilde{{\rm stab}}(hg x)$.
Since the derivative of $f$ at $Lg\gamma_0$ vanishes, we have 
$\xi \in V_{hg\gamma_0 x}$.

Take $\hat{b} \in \widetilde{{\rm stab}}(x)$ arbitrarily, and fix 
$\hat{a} \in\mathbf{l}$ such that 
$\hat{a}+ \sqrt{-1}\hat{b}\in {\rm stab}( x)^{\mathbb{C}}$.
Since ${\rm stab}(hg\gamma_0 x)^{\mathbb{C}} = 
{\rm Ad}_{hg\gamma_0} ({\rm stab}( x)^{\mathbb{C}} )$, there is 
$a+\sqrt{-1}b \in {\rm stab}(hg\gamma_0 x)^{\mathbb{C}}$ 
and $a+\sqrt{-1}b = 
{\rm Ad}_{hg\gamma_0}(\hat{a} + \sqrt{-1}\hat{b})$.
Since $b\in \widetilde{{\rm stab}}(hg\gamma_0 x)$, we have
\begin{eqnarray}
0 = 2 \sqrt{-1}\langle \xi, b\rangle 
&=& \langle \xi, {\rm Ad}_{hg\gamma_0}(\hat{a} 
+ \sqrt{-1}\hat{b}) + ({\rm Ad}_{hg\gamma_0}(\hat{a} 
+ \sqrt{-1}\hat{b}))^*\rangle\nonumber\\
&=& \langle {\rm Ad}_{(hg\gamma_0)^{-1}}\xi, \hat{a} 
+ \sqrt{-1}\hat{b} \rangle + \langle {\rm Ad}_{(hg\gamma_0)^*}\xi
,(\hat{a} + \sqrt{-1}\hat{b})^*\rangle\nonumber\\
&=& \langle {\rm Ad}_{e^{-\sqrt{-1}\xi}}\xi, \hat{a} 
+ \sqrt{-1}\hat{b} \rangle + \langle {\rm Ad}_{e^{\sqrt{-1}\xi}}\xi
,(\hat{a} + \sqrt{-1}\hat{b})^*\rangle\nonumber\\
&=& \langle \xi, \hat{a} 
+ \sqrt{-1}\hat{b} \rangle + \langle \xi
,-\hat{a} + \sqrt{-1}\hat{b}\rangle\nonumber\\
&=& 2\sqrt{-1} \langle \xi, \hat{b} \rangle.\nonumber
\end{eqnarray}
Thus we obtain $\xi \in V_x$.
\end{proof}

\begin{prop}
$\Phi_{x,\zeta}$ has a critical point if and only if
all of the following conditions are satisfied for some $g\in L^{\mathbb{C}}$;
$(i)$ $\Phi_{gx,\zeta}$ is ${\rm Stab}(gx)^{\mathbb{C}}$ invariant, 
$(ii)$ $\lim_{t\to \infty} \Phi_{gx,\zeta}(Le^{\sqrt{-1}t\xi}) = \infty$ 
for any $\xi \in \mathbf{l} - \widetilde{{\rm stab}}(gx)$.
 \label{4.8}
\end{prop}
\begin{proof}
Let $\Phi_{x,\zeta}$ has a critical point $Lg\in L\backslash L^{\mathbb{C}}$ 
for $g\in L^{\mathbb{C}}$. 
Since $\Phi_{gx,\zeta} - R_g^*\Phi_{x,\zeta}$ is a 
constant function, we may suppose $g=1$ by the homogeneity.
Then $\frac{d}{dt}|_{t=0}\Phi_{x,\zeta}(Le^{\sqrt{-1}t\xi}) = 0$ for all 
$\xi\in \mathbf{l}$.
If $\xi \notin \widetilde{{\rm stab}}(x)$, especially 
$\xi \notin {\rm stab}(x)$,
\begin{eqnarray}
\frac{d^2}{dt^2}\Big|_{t=0}
\Phi_{x,\zeta}(Le^{\sqrt{-1}t\xi}) = \|\xi^*_x \|_{\omega}^2 > 0\nonumber
\end{eqnarray}
and 
there exists sufficiently small $\delta >0$ and $\frac{d}{dt}\Phi_{x,\zeta}
(Le^{\sqrt{-1}t\xi}) \ge \delta$ for all $t\ge 1$.
Thus we obtain $\lim_{t\to \infty} \Phi_{x,\zeta}(Le^{\sqrt{-1}t\xi}) = \infty$.
For any $\xi \in \mathbf{l}$, we have 
$\frac{d}{dt}\Phi_{x,\zeta}(Le^{\sqrt{-1}t\xi}) \ge 0$ for any $t>0$
and obtain $\Phi_{x,\zeta}(Le^{\sqrt{-1}\xi}) \ge \Phi_{x,\zeta}(L\cdot 1)$,
hence $\Phi_{x,\zeta}(L\cdot 1)$ is the minimum value of 
$\Phi_{x,\zeta}$, especially $\Phi_{x,\zeta}$ is bounded from below.
Next we show that $\Phi_{x,\zeta}$ is ${\rm Stab}(x)^{\mathbb{C}}$ invariant.
For any $\gamma \in {\rm Stab}(x)^{\mathbb{C}}$, 
$R_{\gamma}^*\Phi_{x,\zeta} - \Phi_{x,\zeta}$ is a constant function.
If we put $R_{\gamma}^*\Phi_{x,\zeta} - \Phi_{x,\zeta} = c_{\gamma} 
\in \mathbb{R}$, then we have $\Phi_{x,\zeta}(L\gamma) 
= \Phi_{x,\zeta}(L\cdot 1) + c_{\gamma}$ and $\Phi_{x,\zeta}(L\gamma^{-1}) 
= \Phi_{x,\zeta}(L\cdot 1) - c_{\gamma}$. 
Since $\Phi_{x,\zeta}(L\cdot 1)$ is a minimum value, 
then $c_{\gamma}$ should be zero, which implies $\Phi_{x,\zeta}$ is 
${\rm Stab}(x)^{\mathbb{C}}$ invariant.

Conversely, assume that the conditions $(i)$-$(ii)$ hold.
It suffices to show that $\Phi_{gx,\zeta}$ has a minimum point 
in $L\backslash L^{\mathbb{C}}$.
To see it, it suffices to see that $\Phi_{gx,\zeta}|_{\exp(V_{gx})}$ has 
a minimum point by applying Lemma \ref{reduce}, where
\begin{eqnarray}
\exp(V_{gx}) &:=& \{L\cdot e^{\sqrt{-1}\xi}\in L\backslash L^{\mathbb{C}}
;\ \xi \in V_{gx}\}\nonumber
\end{eqnarray}
and $\mathbf{l} = \widetilde{{\rm stab}}(gx) \oplus V_{gx}$ is 
the orthogonal decomposition
with respect to the ${\rm Ad}_L$-invariant inner product.
Define a smooth function $F: S(V_{gx})\times \mathbb{R} \to \mathbb{R}$ by
\begin{eqnarray}
F(\xi,t) := \frac{d}{dt}\Phi_{gx,\zeta}(e^{\sqrt{-1}t\xi}),\nonumber
\end{eqnarray}
where $S(V_{gx}):= \{\xi\in V_{gx}; \|\xi \| =1\}$.
Now, recall that $t\mapsto \Phi_{gx,\zeta}(e^{\sqrt{-1}t\xi})$ is convex
and we have assumed that
$\lim_{|t|\to \infty}\Phi_{gx,\zeta}(e^{\sqrt{-1}t\xi}) = \infty$ for 
all $\xi\in S(V_{gx})$.
It implies that there exists a unique $\hat{t}(\xi)\in\mathbb{R}$
for each $\xi \in S(V_{gx})$ such that $F(\xi,\hat{t}(\xi)) = 0$.
Since $\frac{\partial}{\partial t}F(\xi,t) = \|\xi^*_x \|_{\omega}^2 > 0$,
$\hat{t} : S(V_{gx}) \to \mathbb{R}$ is smooth by Implicit Function Theorem.
In particular, $\xi \mapsto \Phi_{gx,\zeta}(e^{\sqrt{-1}\hat{t}(\xi)\xi})$
becomes a smooth function on the compact manifold $S(V_{gx})$, hence it has a
minimum point $\xi_{min}\in S(V_{gx})$.
It is easy to see $\exp (\hat{t}(\xi_{min})\xi_{min}) \in \exp(V_{gx})$ is
a minimum point of $\Phi_{gx,\zeta}|_{\exp(V_{gx})}$.
\end{proof}

Since each ${\rm Stab}(x)^{\mathbb{C}}$-orbit is closed in 
$L\backslash L^{\mathbb{C}}$, the distance on $L\backslash L^{\mathbb{C}}$ 
induces a structure of a metric space on $L\backslash L^{\mathbb{C}}/
{\rm Stab}(x)^{\mathbb{C}}$. 
If $\Phi_{x,\zeta}$ is ${\rm Stab}(x)^{\mathbb{C}}$-invariant, then it
induces a function $\bar{\Phi}_{x,\zeta}:L\backslash L^{\mathbb{C}}/
{\rm Stab}(x)^{\mathbb{C}} \to \mathbb{R}$.

\begin{prop}
$\Phi_{x,\zeta}$ has a critical point if and only if 
$\Phi_{x,\zeta}$ is ${\rm Stab}(x)^{\mathbb{C}}$-invariant,
and $\bar{\Phi}_{x,\zeta}$ is proper and bounded from below.
 \label{proper}
\end{prop}
\begin{proof}
Assume that $\Phi_{x,\zeta}$ has a critical point.
Then the conditions $(i)$-$(ii)$ in Proposition \ref{4.8} are satisfied for 
some $g\in L^{\mathbb{C}}$,
accordingly it suffices to show the properness of $\bar{\Phi}_{gx,\zeta}$.
Define an equivalence relation in ${\rm exp }(V_{gx})$ by $Lg_1\sim
Lg_2$ if $Lg_1$ and $Lg_2$ lie on the same ${\rm Stab}(x)^{\mathbb{C}}$
orbit.
Then the homeomorphism $\bar{\Phi}_{x,\zeta}: {\rm exp }(V_{gx})/\sim \to 
L\backslash L^{\mathbb{C}}/{\rm Stab}(x)^{\mathbb{C}}$ is naturally induced.
Since $\Phi_{x,\zeta}|_{{\rm exp }(V_{gx})}$ is proper, $\bar{\Phi}_{gx,\zeta}$
is also proper.

Conversely, assume that $\Phi_{x,\zeta}$ is ${\rm Stab}(x)^{\mathbb{C}}$-invariant,
and that $\bar{\Phi}_{x,\zeta}$ is proper and bounded from below.
Then the minimizing sequence of $\bar{\Phi}_{x,\zeta}$ always converges,
therefore $\Phi_{x,\zeta}$ has a minimum point.
\end{proof}

\begin{prop}
Assume that $(d\Phi_{x,\zeta})_{L\cdot 1} = 0$.
Let a diffeomorphism $\Psi_L : L\times \mathbf{l} \to L^{\mathbb{C}}$ 
be defined by $\Psi_L(g,\xi) := ge^{\sqrt{-1}\xi}$. 
Then the restriction 
\begin{eqnarray}
\Psi_L|_{{\rm Stab}(x) \times {\rm stab}(x)} 
: {\rm Stab}(x) \times {\rm stab}(x) \to {\rm Stab}(x)^{\mathbb{C}}\nonumber
\end{eqnarray}
is a diffeomorphism.
\label{reductive}
\end{prop}
\begin{proof}
Let $(d\Phi_{x,\zeta})_{L\cdot 1} = 0$.
Take $\gamma\in {\rm Stab}(x)^{\mathbb{C}}$, and put 
$\gamma = ge^{\sqrt{-1}\xi}$ for some $g\in L$ and $\xi\in \mathbf{l}$.
It suffices to show $g\in {\rm Stab}(x)$ and $\xi \in {\rm stab}(x)$.
By the proof of Proposition \ref{4.8}, 
we have $\Phi_{x,\zeta}(L\cdot 1) = \Phi_{x,\zeta}(L\cdot \gamma) = 
\Phi_{x,\zeta}(L\cdot e^{\sqrt{-1}\xi})$.
Since $\Phi_{x,\zeta}$ is geodesically convex and 
$\{L\cdot e^{\sqrt{-1}t\xi}\}_{t\in\mathbb{R}}$ is a geodesic, 
$\Phi_{x,\zeta}$ have to be constant on this geodesic.
Thus we have 
\begin{eqnarray}
0 = \frac{d^2}{dt^2}\Big|_{t=0}\Phi_{x,\zeta}(e^{\sqrt{-1}t\xi}) 
= \|\xi^*_x\|^2,
\nonumber
\end{eqnarray}
which implies $\xi \in {\rm stab}(x)$.
Then we obtain 
\begin{eqnarray}
g=\gamma e^{-\sqrt{-1}t\xi} \in {\rm Stab}(x)^{\mathbb{C}} 
\cap L = {\rm Stab}(x).\nonumber
\end{eqnarray}
\end{proof}

\begin{rem}
{\rm It is shown by Corollary 2.15 in \cite{W} that 
${\rm stab}(x)^{\mathbb{C}}$ are reductive for all $x$ such that 
$(d\Phi_{x,\zeta})_{L\cdot 1} = 0$.}
\end{rem}

Next we assume that there is an $L$-invariant function 
$\varphi\in C^{\infty}(L\backslash L^{\mathbb{C}})$ and satisfies 
$\omega = \sqrt{-1}\partial\bar{\partial}\varphi$.
Then we have
$\Phi_{x,\zeta}(L\cdot e^{\sqrt{-1}\xi}) = 
\varphi(e^{\sqrt{-1}\xi}\cdot x) - \langle \zeta, \xi \rangle + c$ 
for some constant $c$ by the discussion of $(2.6)$ in \cite{HL}. 
Here we may assume $c=0$, since the existence of the critical point does not 
depend on the value of $c$.

\begin{prop}
Assume that $\lim_{t\to \infty}\varphi (e^{\sqrt{-1}t\xi}x)/t = \infty$ holds 
for all $\xi \in \mathbf{l}$ which satisfy 
$\lim_{t\to \infty}\varphi (e^{\sqrt{-1}t\xi}x) = \infty$. 
If $\Phi_{x,\zeta}$ has a critical point, then $\Phi_{x,s\zeta}$ also has 
a critical point for each $s>0$.
\label{perturb}
\end{prop}

\begin{proof}
We may assume the conditions $(i)(ii)$ of Proposition \ref{4.8} 
are satisfied for $\Phi_{x,\zeta}$.
It suffices to show that $\Phi_{x,s\zeta}$ also satisfies $(i)(ii)$.
Since $\Phi_{x,\zeta}$ is ${\rm Stab}(x)^{\mathbb{C}}$ invariant, we have
\begin{eqnarray}
\Phi_{x,s\zeta}(Lg\gamma) &=& s\Phi_{x,\zeta}(Lg\gamma) 
+ (1 -s)\varphi (g\gamma x) \nonumber\\
&=& s\Phi_{x,\zeta}(Lg) + (1 -s)\varphi (gx)
= \Phi_{x,s\zeta}(Lg)\nonumber
\end{eqnarray}
for all $Lg\in L\backslash L^{\mathbb{C}}$ and $\gamma\in
{\rm Stab}(x)^{\mathbb{C}}$, thus $\Phi_{x,s\zeta}$ is 
${\rm Stab}(x)^{\mathbb{C}}$ invariant.
Next we take $\xi\in \mathbf{l} - \widetilde{{\rm stab}}(x)$ and consider the
behavior of $\Phi_{x,s\zeta}(e^{\sqrt{-1}t\xi})$ for $t\to\infty$.
Since $\lim_{t\to \infty}\Phi_{x,\zeta}(e^{\sqrt{-1}t\xi}) = \infty$,
we have $\lim_{t\to \infty}\varphi (e^{\sqrt{-1}t\xi}) = \infty$ or
$-\langle \zeta, \xi\rangle >0$. 
If the latter case occurs, then we obtain $\lim_{t\to \infty}
\Phi_{x,s\zeta}(e^{\sqrt{-1}t\xi}) = \infty$.
Let the former case occur.
From the assumption we have $\varphi (e^{\sqrt{-1}t\xi}x)/t \to \infty$
for $t\to \infty$, 
thus $\lim_{t\to \infty}\Phi_{x,s\zeta}(e^{\sqrt{-1}t\xi}) = \infty$ for 
all $s>0$.
\end{proof}

\section{Main results}\label{mainresult}
\subsection{Correspondence of orbits}\label{corresp}
In this section we prove Theorem \ref{1.1}.
Let $(M,g,I_1,I_2,I_3)$, $H,G,\rho$ and $H_{\rho}$ be as 
in Section \ref{sec1.2}.
First of all, we show the following lemma.

\begin{lem}
Let $\bar{\rho}$ be defined as (\ref{rho1}), and assume it is surjective.
Then a linear map
\begin{eqnarray}
\rho^*|_{{\rm Ann.}\Delta_{\mathbf{g}}} : {\rm Ann.}\Delta_{\mathbf{g}} 
\to {\rm Ann.} \hcapg \nonumber
\end{eqnarray}
is bijective, where 
\begin{eqnarray}
{\rm Ann.}\Delta_{\mathbf{g}} &:=& \{\varphi \in \mathbf{g}^* \oplus 
\mathbf{g}^*;\ \varphi|_{\Delta_{\mathbf{g}}} = 0 \},\nonumber\\
{\rm Ann.} \hcapg &:=& \{\varphi \in 
\mathbf{h}^*;\ \varphi|_{\hcapg} = 0 \}.\nonumber
\end{eqnarray}
\label{linearalg}
\end{lem}
\begin{proof}
The assertion is obvious since $\rho^*|_{{\rm Ann.}\Delta_{\mathbf{g}}}$
is the adjoint map of
\begin{eqnarray}
\bar{\rho}_*: \mathbf{h}/ \hcapg \to (\mathbf{g} \oplus 
\mathbf{g}) / \Delta_{\mathbf{g}},\nonumber
\end{eqnarray}
under the identification $\{ \mathbf{h}/ \hcapg \}^* \cong 
{\rm Ann.}\hcapg$ and $\{ (\mathbf{g} \oplus \mathbf{g}) / 
\Delta_{\mathbf{g}} \}^* \cong {\rm Ann.}\Delta_{\mathbf{g}}$.
\end{proof}

Any $x\in \mu_{\mathbb{C}}^{-1} (\iota^*\zeta_{\mathbb{C}})$
satisfies $\hat{\mu}_{\mathbb{C}}(x) - \zeta_{\mathbb{C}} \in
\ker \iota^* = {\rm Ann.}\hcapgc$ by the definition.
Consequently, there exists a unique $\eta (x) \in 
( \mathbf{g}^{\mathbb{C}} )^*$ such that
\begin{eqnarray}
\hat{\mu}_{\mathbb{C}}(x) - \zeta_{\mathbb{C}} 
= \rho^*(-\eta(x),\eta(x))\nonumber
\end{eqnarray}
by Lemma \ref{linearalg},
which implies $(x,1,\eta(x)) \in 
\sigma_\mathbb{C}^{-1}(\zeta_{\mathbb{C}})$.
Here we identify $N_G = G^{\mathbb{C}}\times \mathbf{g}^{\mathbb{C}}$ and
$( \mathbf{g}^{\mathbb{C}} )^* = \mathbf{g}^{\mathbb{C}} $ by the
${\rm Ad}_G$-invariant $\mathbb{C}$ bilinear form.
Thus we obtain a map $\hat{\psi} : \mu_{\mathbb{C}}^{-1} (\iota^*\zeta_{\mathbb{C}}) \to 
\sigma_\mathbb{C}^{-1}(\zeta_{\mathbb{C}})$ defined by
$\hat{\psi}(x) := (x,1,\eta(x))$, which induces a map
$\psi : \mu_{\mathbb{C}}^{-1} (\iota^*\zeta_{\mathbb{C}}) 
/ \Hcapgc \to \sigma_\mathbb{C}^{-1}(\zeta_{\mathbb{C}}) 
/ H^{\mathbb{C}}$.

\begin{prop}
$\psi$ is well-defined and a homeomorphism.\label{4.4}
\end{prop}
\begin{proof}
First of all we check the well-definedness.
Let $x\in \mu_{\mathbb{C}}^{-1} (\iota^*\zeta_{\mathbb{C}}) $
and $h \in \Hcapgc $. 
Now we may write $\rho (h) = (h_0,h_1) \in G^{\mathbb{C}}\times G^{\mathbb{C}}$,
and suppose $h_0 = h_1$.
Then $hx\in \mu_{\mathbb{C}}^{-1} (\iota^*\zeta_{\mathbb{C}})$
and $\hat{\psi}(hx) = (hx,1,\eta(hx))$, where $\eta(hx)\in \mathbf{g}^{\mathbb{C}} $
 is uniquely defined by $\hat{\mu}_{\mathbb{C}}(hx) - \zeta_{\mathbb{C}} 
= \rho^*(-\eta(hx),\eta(hx))$.
Now we have
\begin{eqnarray}
\hat{\mu}_{\mathbb{C}}(hx) - \zeta_{\mathbb{C}} 
&=& {\rm Ad}^*_{h^{-1}}(\hat{\mu}_{\mathbb{C}}(x) 
- \zeta_{\mathbb{C}} )\nonumber\\
&=& {\rm Ad}^*_{h^{-1}}\rho^*(-\eta(x),\eta(x))\nonumber\\
&=& \rho^*( - {\rm Ad}^*_{h_0^{-1}}\eta(x),
{\rm Ad}^*_{h_0^{-1}}\eta(x)),\nonumber
\end{eqnarray}
where ${\rm Ad}_g^*\in GL( (\mathbf{g}^{\mathbb{C}})^*)$ is defined by
\begin{eqnarray}
\langle {\rm Ad}^*_g y,\xi \rangle 
:= \langle y,{\rm Ad}_g\xi \rangle\nonumber
\end{eqnarray}
for $y\in (\mathbf{g}^{\mathbb{C}})^*,\xi\in \mathbf{g}^{\mathbb{C}}, 
g\in G^{\mathbb{C}}$,
hence $\eta(hx) = {\rm Ad}^*_{h_0^{-1}}\eta(x)$ holds
by the uniqueness.
Since ${\rm Ad}^*_{h_0^{-1}}$ corresponds to ${\rm Ad}_{h_0}$ under the
identification $(\mathbf{g}^{\mathbb{C}})^* \cong \mathbf{g}^{\mathbb{C}}$,
we obtain $\hat{\psi}(hx) = h\hat{\psi}(x)$.

Next we show the injectivity. Take $x,x'\in \mu_{\mathbb{C}}^{-1} 
(\iota^*\zeta_{\mathbb{C}}) $ to be $\hat{\psi}(x) = h\hat{\psi}(x')$ for some 
$h\in H^{\mathbb{C}}$.
Then $(x,1,\eta(x)) = (hx', h_0h_1^{-1},{\rm Ad}_{h_0} \eta(x'))$, accordingly
we obtain $h_0h_1^{-1} = 1$ which implies $h\in \Hcapgc$.

The surjectivity is shown by constructing the inverse map of $\psi$
as follows.
Take $(x,Q,\eta) \in \sigma_\mathbb{C}^{-1}(\zeta_{\mathbb{C}}) $
arbitrarily. 
From Section \ref{3.2}, we have
\begin{eqnarray}
\sigma_{\mathbb{C}} (x,Q,\eta) &=& \hat{\mu}_{\mathbb{C}} (x)
 + \rho^*(\nu (Q,\eta))\nonumber\\
&=& \hat{\mu}_{\mathbb{C}} (x) + 
\rho^*(\eta, -{\rm Ad}_{Q^{-1}}\eta) = \zeta_{\mathbb{C}}\label{eq4.1.1}
\end{eqnarray}
From the surjectivity of the map (\ref{rho1}), there exist some 
$h\in H^{\mathbb{C}}$ such that $h_0h_1^{-1} = Q$.
Then $\rho(h)^{-1} = (h_0^{-1}, h_0^{-1}Q) \in H^{\mathbb{C}}$,
and we have
\begin{eqnarray}
\rho^*(\eta, -{\rm Ad}_{Q^{-1}}\eta) 
&=& \rho^*{\rm Ad}_{\rho(h)}({\rm Ad}_{h_0^{-1}}\eta, 
-{\rm Ad}_{h_0^{-1}}\eta) \nonumber\\
&=& {\rm Ad}_h\rho^*({\rm Ad}_{h_0^{-1}}\eta, 
-{\rm Ad}_{h_0^{-1}}\eta).\label{eq4.1.2}
\end{eqnarray}
By combining (\ref{eq4.1.1})(\ref{eq4.1.2}), we obtain
\begin{eqnarray}
\hat{\mu}_{\mathbb{C}} (h^{-1}x) + 
\rho^*({\rm Ad}_{h_0^{-1}}\eta, -{\rm Ad}_{h_0^{-1}}\eta) 
= \zeta_{\mathbb{C}},\nonumber
\end{eqnarray}
which means $h^{-1}x \in \mu_{\mathbb{C}}^{-1} (\iota^*\zeta_{\mathbb{C}})$
and 
\begin{eqnarray}
\hat{\psi}(h^{-1}x) = (h^{-1}x,1,{\rm Ad}_{h_0^{-1}}\eta) 
= h^{-1}(x,Q,\eta).\nonumber
\end{eqnarray}
Thus we have the surjectivity of $\psi$.
Here, we can take $h$ depending on $Q$ continuously in local,
therefore the inverse of $\psi$ becomes continuous.
\end{proof}

We can give group isomorphisms between the stabilizers as follows.
Let
\begin{eqnarray}
{\rm Stab}(x)^{\mathbb{C}} &:=& \{g\in \Hcapgc;\ gx = x\},\nonumber\\
{\rm Stab}(x,Q,\eta)^{\mathbb{C}} &:=& \{g\in H^{\mathbb{C}};\ 
g(x,Q,\eta) = (x,Q,\eta)\}.\nonumber
\end{eqnarray}
Then it is easy to check that the inclusion 
${\rm Stab}(x)^{\mathbb{C}} \hookrightarrow {\rm Stab}(\hat{\psi}(x))^{\mathbb{C}}$
is surjective, hence we obtain a Lie group isomorphism
\begin{eqnarray}
{\rm Stab}(x)^{\mathbb{C}} \cong {\rm Stab}(\hat{\psi}(x))^{\mathbb{C}}.\label{stabc}
\end{eqnarray}

\subsection{Correspondence of stability}\label{stability}
Put
\begin{eqnarray}
\mu_{\mathbb{C}}^{-1}(\iota^*\zeta_{\mathbb{C}})_{\iota^*\zeta_1} &=& 
\{ x\in \mu_{\mathbb{C}}^{-1}(\iota^*\zeta_{\mathbb{C}});\ 
\Phi_{x,\iota^*\zeta_1}\ {\rm has\ a\ critical\ point} \},\nonumber\\
\sigma_\mathbb{C}^{-1}(\zeta_{\mathbb{C}})_{\zeta_1} &=& 
\{ (x,Q,\eta)\in \sigma_\mathbb{C}^{-1}(\zeta_{\mathbb{C}});\ 
\Phi_{(x,Q,\eta),\zeta_1}\ {\rm has\ a\ critical\ point} \}.\nonumber
\end{eqnarray}
In this subsection we prove that $\psi$ is a bijection from 
$\mu_{\mathbb{C}}^{-1}
(\iota^*\zeta_{\mathbb{C}})_{\iota^*\zeta_1} / \Hcapgc$ to 
$\sigma_\mathbb{C}^{-1}(\zeta_{\mathbb{C}})_{\zeta_1} 
/ H^{\mathbb{C}}$
by using the results in Section \ref{hamilton}.

Along Section \ref{hamilton}, we define geodesically convex functions 
\begin{eqnarray}
\Phi_{x,\zeta_1} : H \backslash H^{\mathbb{C}} \to\mathbb{R},
\quad \ 
\Phi_{x,\iota^*\zeta_1} : \Hcapg\backslash \Hcapgc 
\to\mathbb{R},
\quad \ \Phi_{(x,Q,\eta),\zeta_1} : H \backslash H^{\mathbb{C}} 
\to\mathbb{R},\nonumber
\end{eqnarray}
for $x\in M$ and $(x,Q,\eta) \in M\times N_G$, 
corresponding to the moment maps ${m} = \hat{\mu}_1,\mu_1,\sigma_1$, 
respectively.
Since $\Hcapg$ is a closed subgroup of $H$, $\Hcapg\backslash \Hcapgc$ is 
naturally embedded in $H \backslash H^{\mathbb{C}}$. 
Then we have $\Phi_{x,\iota^*\zeta_1} (\Hcapg h) = 
\Phi_{x,\zeta_1} (H h)$ for all $h\in \Hcapgc$.
Moreover we may write $\Phi_{(x,Q,\eta),\zeta_1}(H \hat{h}) 
= \Phi_{x,\zeta_1} (H \hat{h}) 
+ \mathcal{E}(\hat{h}_0Q\hat{h}_1^{-1},{\rm Ad}_{\hat{h}_0}\eta)$ 
for all $\hat{h} \in H^{\mathbb{C}}$ from Proposition \ref{potential2} and 
$(2.6)$ in \cite{HL}.

\begin{prop}
Let $x\in \mu_{\mathbb{C}}^{-1}(\iota^*\zeta_{\mathbb{C}})_{\iota^*\zeta_1}$.
Then $\hat{\psi}(x)\in \sigma_\mathbb{C}^{-1}(\zeta_{\mathbb{C}})_{\zeta_1}$.
\label{psi}
\end{prop}
\begin{proof}
It suffices to show that $\Phi_{\hat{\psi}(x),\zeta_1}$ has a critical point if
$\Phi_{x,\iota^*\zeta_1}$ has a critical point.

First of all, it is easy to check that 
$\Phi_{\hat{\psi}(x),\zeta_1}$ is ${\rm Stab}(\hat{\psi}(x))^{\mathbb{C}}$ invariant, since 
${\rm Stab}(x)^{\mathbb{C}} = {\rm Stab}(\hat{\psi}(x))^{\mathbb{C}}$ and
$\Phi_{x,\iota^*\zeta_1}$ is 
${\rm Stab}(x)^{\mathbb{C}}$ invariant.

Next we take $\hat{\xi} \in \mathbf{h}$, put $\rho_*(\hat{\xi}) 
= (\hat{\xi}_0, \hat{\xi}_1)$ and consider the behavior of
$\Phi_{\hat{\psi}(x),\zeta_1}(He^{\sqrt{-1}t\hat{\xi}})$ for $t\to \infty$.
Since $\Phi_{x,\zeta_1}$ is geodesically convex, there is a constant 
$c_{\hat{\xi}}\mathbb{R}$ and $\liminf_{t\to +\infty}\frac{d}{dt}
\Phi_{x,\zeta_1} (He^{\sqrt{-1}t\hat{\xi}}) \ge c_{\hat{\xi}}$. 
Then we have an inequality
$\Phi_{x,\zeta_1} (He^{\sqrt{-1}t\hat{\xi}}) \ge c_{\hat{\xi}} t - N_1$
for all $t\in\mathbb{R}$, for some sufficiently large $N_1$.
If $\hat{\xi}_0 \neq \hat{\xi}_1$, then
\begin{eqnarray}
\Phi_{\hat{\psi}(x),\zeta_1}(He^{\sqrt{-1}t\hat{\xi}}) &\ge& \mathcal{E}
(e^{\sqrt{-1}t\hat{\xi}_0}
e^{ - \sqrt{-1}t\hat{\xi}_1}, {\rm Ad}_{e^{\sqrt{-1}t\hat{\xi}_0}}\eta(x))
+ c_{\hat{\xi}} t - N_1\nonumber\\
&\ge& \min_{h\in P(e^{2\sqrt{-1}t\hat{\xi}_0}, e^{ - 2\sqrt{-1}t\hat{\xi}_1})} 
\int_0^1 \frac{1}{4}\| h' \|_h^2 
+ c_{\hat{\xi}} t - N_1\nonumber\\
&\ge& {\rm dist}_{G\backslash G^{\mathbb{C}}} 
(e^{2\sqrt{-1}t\hat{\xi}_0}, e^{ 2\sqrt{-1}t\hat{\xi}_1} )^2
+ c_{\hat{\xi}} t - N_1\nonumber
\end{eqnarray}
Now $G\backslash G^{\mathbb{C}}$ is an Hadamard manifold, therefore
the function 
\begin{eqnarray}
t\mapsto {\rm dist}_{G\backslash G^{\mathbb{C}}} 
(e^{2\sqrt{-1}t\hat{\xi}_0}, e^{ 2\sqrt{-1}t\hat{\xi}_1} )\nonumber
\end{eqnarray}
is convex.
Since $\hat{\xi}_0 \neq \hat{\xi}_1$, there exists a positive constant $N_2>0$ 
and
\begin{eqnarray}
{\rm dist}_{G\backslash G^{\mathbb{C}}} 
(e^{2\sqrt{-1}t\hat{\xi}_0}, e^{ 2\sqrt{-1}t\hat{\xi}_1} )^2 \ge N_2t^2\nonumber
\end{eqnarray}
for $t\ge 1$, and we obtain $\Phi_{\hat{\psi}(x),\zeta_1}(He^{\sqrt{-1}t\hat{\xi}}) 
\to \infty$ for $t\to \infty$.

If $\hat{\xi}_0 = \hat{\xi}_1$, then $\hat{\xi} \in \hcapg$.
In this case we have
\begin{eqnarray}
\Phi_{\hat{\psi}(x),\zeta_1}(He^{\sqrt{-1}t\hat{\xi}}) \ge \Phi_{x,\iota^*\zeta_1}
(\Hcapg e^{\sqrt{-1}t\hat{\xi}}) \to \infty\nonumber
\end{eqnarray}
for $t\to \infty$, if we take $\hat{\xi} \notin \widetilde{{\rm stab}}(\hat{\psi}(x))
= \widetilde{{\rm stab}}(x)$, where the $\widetilde{{\rm stab}}$ is defined in 
the next section. 
Thus $\Phi_{\hat{\psi}(x),\zeta_1}$ has a critical value by Proposition \ref{4.8}.
\end{proof}

Next we show the converse correspondence.
From now on, we assume that there is an $H$-invariant global K\"ahler potential
$\varphi: M\to \mathbb{R}$ of $(M,I_1,\omega_1)$, 
then we have 
\begin{eqnarray}
\Phi_{x,\iota^*\zeta_1} (\Hcapg e^{\sqrt{-1}\xi}) 
&=& \varphi(e^{\sqrt{-1}\xi} x)
- \langle \iota^*\zeta_1, \xi\rangle + const.\ ,\nonumber\\
\Phi_{(x,Q,\eta),\zeta_1}(He^{\sqrt{-1}\hat{\xi}}) 
&=& \varphi(e^{\sqrt{-1}\hat{\xi}} x) + \mathcal{E}(Q,\eta)
- \langle \zeta_1, \hat{\xi} \rangle + const.\ ,\nonumber
\end{eqnarray}
where $\xi\in \hcapg$ and $\hat{\xi} \in \mathbf{h}$.
Here we may assume the constant in the right hand sides of equalities 
are equal to $0$.

\begin{prop}
Assume that there exists a smooth function $q:\mathbb{R} \to \mathbb{R}$
such that $\|\eta(x) \|^2 \le q(\varphi(x))$ and $q'(\varphi(x))\ge 0$ 
for any $x\in M$. 
Suppose that
if $\varphi(e^{\sqrt{-1}t\xi}\cdot x) \to \infty$ for $t\to \infty$ then 
$\lim_{t\to \infty}\varphi(e^{\sqrt{-1}t\xi}\cdot x)/t = \infty$ for any 
$\xi \in \hcapg$.
If $\hat{\psi}(x)\in \sigma_\mathbb{C}^{-1}(\zeta_{\mathbb{C}})_{\zeta_1}$, then
$x\in \mu_{\mathbb{C}}^{-1}(\iota^*\zeta_{\mathbb{C}})_{\iota^*\zeta_1}$.
\label{converse}
\end{prop}
\begin{proof}
Assume that $\Phi_{\hat{\psi}(x),\zeta_1}$ has a critical point.
From Proposition \ref{proper}, $\Phi_{\hat{\psi}(x),\zeta_1}$ is 
${\rm Stab}(\hat{\psi}(x))^{\mathbb{C}}$ invariant and the induced map
\begin{eqnarray}
\bar{\Phi}_{\hat{\psi}(x),\zeta_1}:H\backslash H^{\mathbb{C}}/
{\rm Stab}(\hat{\psi}(x))^{\mathbb{C}} \to \mathbb{R}\nonumber
\end{eqnarray}
is proper, and bounded from below.
Since $H_{\rho}\backslash H_{\rho}^{\mathbb{C}}/
{\rm Stab}(x)^{\mathbb{C}}$ is a closed subset of 
$H\backslash H^{\mathbb{C}}/{\rm Stab}(\hat{\psi}(x))^{\mathbb{C}}$,
$F:=\bar{\Phi}_{\hat{\psi}(x),\zeta_1}
|_{H_{\rho}\backslash H_{\rho}^{\mathbb{C}}/{\rm Stab}(x)^{\mathbb{C}}}$
is also proper and bounded below.

If $(Q,\eta):= (1,\eta(g x))$ for $g\in H_{\rho}^{\mathbb{C}}$, 
we have an upper estimate
\begin{eqnarray}
\mathcal{E}(1,\eta(gx)) 
\le \|\eta(x)\|^2 
\le q(\varphi (gx)),\nonumber
\end{eqnarray}
by taking a path $h\in P(1,1)$ to be $h(s)=1$.
Hence if we take $\xi \in \hcapg$, then
\begin{eqnarray}
\Phi_{\hat{\psi}(x),\zeta_1}(H e^{\sqrt{-1}\xi}) 
\le \Phi_{x,\iota^*\zeta_1}(\Hcapg e^{\sqrt{-1}\xi}) 
+ q(\varphi (e^{\sqrt{-1}\xi} x))
=: \hat{F}_+(\Hcapg e^{\sqrt{-1}\xi}).\nonumber
\end{eqnarray}
Now $\hat{F}_+$ induces a function $F_+: H_{\rho}\backslash 
H_{\rho}^{\mathbb{C}}/{\rm Stab}(x)^{\mathbb{C}} \to \mathbb{R}$, 
which satisfies $F_+ \ge F$, 
therefore $F_+$ is also proper and bounded from below.
Thus $\hat{F}_+$ has a minimum point $e^{\sqrt{-1}\xi}\in 
H_{\rho}\backslash H_{\rho}^{\mathbb{C}}$, and have
\begin{eqnarray}
0 = (d \hat{F}_+)_{e^{\sqrt{-1}\xi}} = (1 + q'(\varphi(e^{\sqrt{-1}\xi} x)))
\mu (e^{\sqrt{-1}\xi} x) - \iota^*\zeta_1.\nonumber
\end{eqnarray}
Now we have shown that $\Phi_{x,s\cdot\iota^*\zeta_1}$ has a critical point
if we put $s=1 + q'(\varphi(e^{\sqrt{-1}\xi} x))$,
hence $\Phi_{x,\iota^*\zeta_1}$ also has a critical point by 
Proposition \ref{perturb}.
\end{proof}
\begin{rem}
{\rm The assumption of Proposition \ref{converse} is always satisfied
if $M$ is the quaternionic vector space $\mathbb{H}^N$ with Euclidean
metric, and $H\subset Sp(N)$ acts on $M$ linearly.}
\end{rem}

\subsection{Proof of the main theorem}\label{pf.of1.1}
\begin{prop}
${\rm Stab}(x)\subset \Hcapg$ and ${\rm Stab}(y) \subset H$ are 
isomorphic as Lie groups for any $x\in \mu^{-1} (\iota^*\zeta)$ and 
$y\in \sigma^{-1}(\zeta)$ satisfying $yH = \psi(x\Hcapg)$.
\label{stab}
\end{prop}
\begin{proof}
The assertion follows directly from Proposition \ref{reductive} and 
the isomorphism (\ref{stabc}).
\end{proof}

\begin{proof}[Proof of Theorem \ref{1.1}]
Define an open subsets $\mu_{\mathbb{C}}^{-1}
(\iota^*\zeta_{\mathbb{C}})_{\iota^*\zeta_1}^{ss} \subset 
\mu_{\mathbb{C}}^{-1}(\iota^*\zeta_{\mathbb{C}})$ and 
$\sigma_\mathbb{C}^{-1}(\zeta_{\mathbb{C}})_{\zeta_1}^{ss} \subset 
\sigma_\mathbb{C}^{-1}(\zeta_{\mathbb{C}})$ by
\begin{eqnarray}
\mu_{\mathbb{C}}^{-1} (\iota^*\zeta_{\mathbb{C}})_{\iota^*\zeta_1}^{ss}
&:=& \{x\in \mu_{\mathbb{C}}^{-1}(\iota^*\zeta_{\mathbb{C}});\ 
\overline{\Hcapgc\cdot x} \cap \mu_1^{-1}(\iota^*\zeta_1) \neq \varnothing\},
\nonumber\\
\sigma_\mathbb{C}^{-1}(\zeta_{\mathbb{C}})_{\zeta_1}^{ss} 
&:=& \{y\in \mu_{\mathbb{C}}^{-1}(\iota^*\zeta_{\mathbb{C}});\ 
\overline{H^{\mathbb{C}}\cdot y} \cap \sigma_1^{-1}(\zeta_1) \neq \varnothing\}.
\nonumber
\end{eqnarray}
Then the naturally induced maps
\begin{eqnarray}
\mu^{-1} (\iota^*\zeta) / \Hcapg \to \mu_{\mathbb{C}}^{-1} 
(\iota^*\zeta_{\mathbb{C}})_{\iota^*\zeta_1}^{ss} // \Hcapgc, \quad
\sigma^{-1}(\zeta) / H \to 
\sigma_\mathbb{C}^{-1}(\zeta_{\mathbb{C}})_{\zeta_1}^{ss} // H^{\mathbb{C}}
\nonumber
\end{eqnarray}
gives an isomorphisms as complex analytic spaces 
by the main theorem in \cite{HL}, 
where $//$ is the categorical quotient.
Moreover $\mu_{\mathbb{C}}^{-1} 
(\iota^*\zeta_{\mathbb{C}})_{\iota^*\zeta_1}^{ss} $ and 
$\sigma_\mathbb{C}^{-1}(\zeta_{\mathbb{C}})_{\zeta_1}^{ss} $
are the minimal open subsets of $\mu_{\mathbb{C}}^{-1} 
(\iota^*\zeta_{\mathbb{C}})$ and 
$\sigma_\mathbb{C}^{-1}(\zeta_{\mathbb{C}})$ 
containing $\mu_{\mathbb{C}}^{-1} 
(\iota^*\zeta_{\mathbb{C}})_{\iota^*\zeta_1}$ and
$\sigma_\mathbb{C}^{-1}(\zeta_{\mathbb{C}})_{\zeta_1}$, respectively.
Then $\psi$ gives a bijective map 
$\mu_{\mathbb{C}}^{-1} 
(\iota^*\zeta_{\mathbb{C}})_{\iota^*\zeta_1}^{ss} // \Hcapgc \to 
\sigma_\mathbb{C}^{-1}(\zeta_{\mathbb{C}})_{\zeta_1}^{ss} // H^{\mathbb{C}}$
by Propositions \ref{4.4}, \ref{psi}, and \ref{converse}.
Moreover it is biholomorphic since $\hat{\psi}$ is obviously holomorphic and 
the inverse of $\psi$ is also holomorphically defined in the proof of 
Proposition \ref{4.4}.

Let $(I^{\iota^*\zeta}_1,I^{\iota^*\zeta}_2,I^{\iota^*\zeta}_3)$ be the
hypercomplex structure on $\mu^{-1} (\iota^*\zeta)/ \Hcapg$ induced from 
$(I_1,I_2,I_3)$ on $M$. 
Similarly, let $(I^{\zeta}_1,I^{\zeta}_2,I^{\zeta}_3)$ be 
the hypercomplex structure on $\sigma^{-1}(\zeta) / H$ induced from 
$(I_1\times I_{G,1},I_2\times I_{G,2},I_3\times I_{G,3})$ on $M\times N_G$.
Moreover, let $\omega^{\iota^*\zeta}_i$ and $\omega^{\zeta}_i$ be the 
corresponding K\"ahler forms.

If $\Hcapg$ acts on $\mu^{-1} (\iota^*\zeta)$ freely, 
then $H$ also acts on $\sigma^{-1}(\zeta)$ freely from Proposition \ref{stab},
hence $\mu^{-1} (\iota^*\zeta)/ \Hcapg$ and $\sigma^{-1}(\zeta) / H$ become 
smooth \hK manifolds by \cite{HKLR}. 
Since $M$ and $M\times N_G$ are complete, 
$\mu^{-1} (\iota^*\zeta)/ \Hcapg$ and $\sigma^{-1}(\zeta) / H$ are 
complete, too.
See \cite{K2} for the completeness of $N_G$.

The equality $\psi^*(\omega^{\zeta}_2 
+\sqrt{-1}\omega^{\zeta}_3) = \omega^{\iota^*\zeta}_2 + 
\sqrt{-1}\omega^{\iota^*\zeta}_3$ follows directly from the definition of $\hat{\psi}$ 
in Section \ref{corresp} and the fact that any fiber of $T^*G^{\mathbb{C}}$ 
are holomorphic Lagrangian submanifolds.

Next we show the corresponding of K\"ahler classes. 
For each $y\in S^2 = \{y=(y_1,y_2,y_3) \in \mathbb{R}^3;\ 
y_1^2 +y_2^2 +y_3^2 =1\}$, put 
\begin{eqnarray}
I^{\iota^*\zeta}_y := \sum_{i=1}^{3}y_iI^{\iota^*\zeta}_i,\quad 
I^{\zeta}_y := \sum_{i=1}^{3}y_iI^{\zeta}_i.\nonumber
\end{eqnarray}
Now we take $y',y''\in S^2$ such that $\{y,y',y''\}$ is the 
orthonormal basis of $\mathbb{R}^3$ with the positive orientation.
Then we can apply Theorem \ref{1.1} for the complex structure 
$I^{\iota^*\zeta}_y$ and $I^{\zeta}_y$, and obtain a biholomorphism 
$\psi_y$. 
Thus we obtain a continuous family of diffeomorphisms $\{\psi_y\}_y$
parametrized by $y\in S^2$.
Since $S^2$ is connected, the induced maps 
\begin{eqnarray}
\psi_y^*: H^2(\sigma^{-1}(\zeta) / H,\mathbb{R})
\to H^2(\mu^{-1} (\iota^*\zeta)/ \Hcapg,\mathbb{R})\nonumber
\end{eqnarray}
does not depend on $y\in S^2$.
Since each $\psi_y$ identifies the holomorphic symplectic forms 
with respect to $I^{\iota^*\zeta}_y$ and $I^{\zeta}_y$, 
therefore $\psi_y^*[\omega^{\zeta}_i] = [\omega^{\iota^*\zeta}_i]$ 
holds for $i=1,2,3$.
\end{proof}

Finally, we show the correspondence of the parameter spaces of 
two \hK quotients.
\begin{prop}
Let $\zeta,\zeta'\in \ImH \otimes Z_{H}$ satisfies 
$\iota^*\zeta = \iota^*\zeta'$.
Then \hK quotients $\sigma^{-1}(\zeta) / H$ and $\sigma^{-1}(\zeta') / H$ are 
canonically identified.
\end{prop}
\begin{proof}
Take $\zeta,\zeta'\in \ImH \otimes Z_{H}$ such that 
$\iota^*\zeta = \iota^*\zeta'$.
Then $\zeta' - \zeta \in \ImH \otimes {\rm Ann}. \hcapg$, 
and there exists a unique $A = (A_1,A_2,A_3)\in \ImH \otimes \mathbf{g}^*$ 
such that 
$\rho^*(A, -A) = \zeta' - \zeta$ by Lemma \ref{linearalg}.
Define $\hat{A}\in \mathcal{N}_G$ by $\hat{A}(t):= (0,A_1,A_2,A_3)$ for all 
$t\in [0,1]$.
Then a $C^{\infty}$ map $\sigma^{-1}(\zeta) \to \sigma^{-1}(\zeta')$ 
defined by $(x,[T])\mapsto (x,[T + \hat{A}])$ gives an isomorphism 
$\sigma^{-1}(\zeta)/ H \to \sigma^{-1}(\zeta')/ H$.
\end{proof}

\section{Examples}
Here we raise some examples which Theorem \ref{1.1} can be applied to.
\subsection{Hilbert schemes of $k$ points on $\mathbb{C}^2$}\label{hilb}
Here we apply the main results obtained in the previous sections
to the case of
\begin{eqnarray}
M = {\rm End} (\mathbb{C}^k) \oplus {\rm End} (\mathbb{C}^k) 
\oplus \mathbb{C}^k \oplus (\mathbb{C}^k)^*\nonumber
\end{eqnarray}
$G = U(k)$, $H = U(k)\times U(k)$ and $\rho = {\rm id}: H \to G\times G$.
Here, $H$-action on $M$ is defined by
$(g_0,g_1)\cdot (A,B,p,q) := (g_0Ag_1^{-1}, g_1Bg_0^{-1},g_0 p, qg_0^{-1})$
for $g_0,g_1\in U(k)$, $A,B\in {\rm End} (\mathbb{C}^k)$, $p\in \mathbb{C}^k$
and $q\in (\mathbb{C}^k)^*$.
According to \cite{Na1}, $Z_{\Hcapg} \cong\mathbb{R}$ and 
$\mu^{-1}(\iota^*\zeta)/\Hcapg$ is a smooth hyper-K\"ahler manifolds 
diffeomorphic to a crepant resolution of $(\mathbb{C}^2)^k/\mathcal{S}_k$
if $\iota^*\zeta \in {\rm Im}\mathbb{H}$ is given by
$\iota^*\zeta = (t,0,0)$ for $t\neq 0$ in this situation.
Here, $\mathcal{S}_k$ is the symmetric group acting on $(\mathbb{C}^2)^k$.
If $\iota^*\zeta = 0$, then $\mu^{-1}(0)/\Hcapg$ is isometric to
$(\mathbb{C}^2)^k/\mathcal{S}_k$ with Euclidean metric.
Then we have a family of smooth hyper-K\"ahler manifolds 
$\sigma^{-1}(\zeta)/ H$ which are biholomorphic to
$\mu^{-1}(\iota^*\zeta)/\Hcapg$. 
In particular, we can study $\sigma^{-1}(0)/ H$ which gives a 
singular hyper-K\"ahler metric on $(\mathbb{C}^2)^k/\mathcal{S}_k$ as follows.

\begin{thm}
Let $M,H,G,\rho$ be as above. Then $\sigma^{-1}(0)/ H$ is isometric to
$(\mathbb{C}_{Taub-NUT}^2)^k/\mathcal{S}_k$ on their regular parts, where
$\mathbb{C}_{Taub-NUT}^2$ is Taub-NUT space.
\label{taubnuthilb}
\end{thm}

Before the proof of Theorem \ref{taubnuthilb}, we see that $N_L$ is identified
with the open subset of $L\times \mathbf{l}^3$ as follows by \cite{DS}, for any compact Lie group
$L$.
Let $T\in \mathcal{N}_L$ and $f:[0,1]\to L$ be the solution of 
the initial value problem
\begin{eqnarray}
{\rm Ad}_{f(s)}T_0(s) + f(s)\frac{d}{ds}f(s)^{-1} &=& 0,\nonumber\\
f(1) &=& 1. \nonumber
\end{eqnarray}
Then a $C^{\infty}$ map $\phi: N_L \to L\times \mathbf{l}^3$ is defined by
\begin{eqnarray}
\phi([T]) := (f(0)^{-1},T_1(1),T_2(1),T_3(1)).\nonumber
\end{eqnarray}
$\phi$ is an diffeomorphism from $N_L$ to an open subset of 
$L\times \mathbf{l}^3$.
In particular, $\phi$ is surjective and an isomorphism of hyper-K\"ahler 
manifolds if $L$ is a torus, 
therefore we may assume $N_{T^k} = T^k\times \mathbb{R}^{3k}$.

Next we begin the proof of Theorem \ref{taubnuthilb}.
The inclusion $T^k\subset U(k)$ which is given by
\begin{eqnarray}
T^k = \{ {\rm diag}(g_1,\cdots,g_k)\in U(k);\ g_1,\cdots,g_k\in S^1\}\nonumber
\end{eqnarray}
induces an embedding $N_{T^k}\subset N_{U(k)}$.
Now we put
\begin{eqnarray}
M_0 &:=& \{(A,B,0,0)\in M;\ A = {\rm diag}(a_1,\cdots,a_k),\ 
B = {\rm diag}(b_1,\cdots,b_k)\} \nonumber\\
&\cong & \mathbb{C}^k \oplus \mathbb{C}^k, \nonumber
\end{eqnarray}
then $\hat{M}_0 := M_0\times N_{T^k}$ is a hyper-K\"ahler submanifold of 
$\hat{M} := M\times N_{U(k)}$.

Let a closed sub group $H_0\subset H$ be generated by
\begin{eqnarray}
\{(g\chi,\chi)\in U(k)\times U(k);\ g\in T^k,\ \chi\in\mathcal{S}_k \},\nonumber
\end{eqnarray}
then $H_0$ is isomorphic to $\mathcal{S}_k\ltimes T^k$.
Then, $H_0$-action is closed on $\hat{M_0}$, and we obtain the hyper-K\"ahler 
moment map $\sigma_0:=\iota_0^*\circ \sigma |_{\hat{M_0}}:\hat{M_0} \to 
{\rm Im}\mathbb{H}\otimes \mathbf{h}_0^*$, where $\iota_0^* : \mathbf{h}^*\to
\mathbf{h}_0^*$ is the adjoint map of the inclusion $\mathbf{h}_0 
\hookrightarrow \mathbf{h}$.
Here, $\mathbf{h}_0 = \mathbf{u}(k) \oplus \{ 0 \}$ is the Lie algebra of $H_0$.

\begin{lem}
We have $\sigma_0^{-1}(0) = (\sigma |_{\hat{M}_0})^{-1}(0)$,
and the naturally induced map $\sigma_0^{-1}(0)/H_0 \to \sigma^{-1}(0)/H$
is injective.
\end{lem}
\begin{proof}
For $x= (A,B,0,0)\in M$, we have
\begin{eqnarray}
\hat{\mu}(x) &=& (\sqrt{-1}(B^*B - AA^*), \sqrt{-1}(-AB - B^*A^*), -AB + B^*A^*)
\nonumber\\
&\ &\quad \oplus (\sqrt{-1}(A^*A - BB^*), \sqrt{-1}(BA + A^*B^*), 
BA - A^*B^*)\nonumber\\
&\in &{\rm Im} \mathbb{H}\otimes (\mathbf{u}(k) \oplus \mathbf{u}(k)),\nonumber
\end{eqnarray}
where $\mathbf{u}(k)$ is the Lie algebra of $U(k)$, and we identify 
$\mathbf{u}(k) \cong \mathbf{u}(k)^*$ by the bilinear form 
$(u,v)\mapsto {\rm tr} (uv^*)$.
If $x\in M_0$, we can put
\begin{eqnarray}
A = {\rm diag}(a_1,\cdots,a_k),\ B = {\rm diag}(b_1,\cdots,b_k),\nonumber
\end{eqnarray}
and we obtain
\begin{eqnarray}
\hat{\mu}(x) = -\sqrt{-1}\tau(x) \oplus \sqrt{-1}\tau(x)
\in {\rm Im} \mathbb{H}\otimes (\mathbf{t}^k \oplus \mathbf{t}^k),\nonumber
\end{eqnarray}
where $\mathbf{t}^k := {\rm Lie} (T^k) \subset \mathbf{u}(k)$,
and $\tau = (\tau_1,\tau_2,\tau_3) : M_0 \to {\rm Im} \mathbb{H}\otimes 
\mathbb{R}^k$ is the hyper-K\"ahler moment map with respect to the
tri-Hamiltonian $T^k$-action on $M_0$
defined by
\begin{eqnarray}
\tau_1(x) &=& {\rm diag}(|a_1|^2 - |b_1|^2,\cdots, |a_k|^2 - |b_k|^2), 
\nonumber\\
\tau_2(x) &=& {\rm diag}(2{\rm Re} (a_1b_1),\cdots, 2{\rm Re} (a_kb_k)), 
\nonumber\\
\tau_3(x) &=& {\rm diag}(2{\rm Im} (a_1b_1),\cdots, 2{\rm Im} (a_kb_k)).
\nonumber
\end{eqnarray}
Under the identification $(\theta,y)\in T^k\times \mathbb{R}^{3k} = N_{T^k}$,
we obtain $\rho^*(\nu(\theta,y)) = \sqrt{-1}(y,-y)$.
Thus we have $\sigma (x,\theta,y) = \sqrt{-1}(-\tau (x) + y,\tau (x) - y)$
and $\sigma_0(x,\theta,y) = \sqrt{-1}(-\tau (x) + y)$ for 
$(x,t,y)\in \hat{M}_0$,
which implies $\sigma_0^{-1}(0) = (\sigma |_{\hat{M}_0})^{-1}(0)$.
Then we obtain $\sigma_0^{-1}(0)/H_0 \to \sigma^{-1}(0)/H$ by the inclusion
$\sigma_0^{-1}(0) \subset \sigma^{-1}(0)$.

Next we show the injectivity of $\sigma_0^{-1}(0)/H_0 \to \sigma^{-1}(0)/H$.
Let $(x,\theta,y)\in \sigma_0^{-1}(0)$ and $(g_0,g_1)\in H = U(k) \times U(k)$
satisfy $(g_0,g_1)\cdot (x,\theta,y)\in \sigma_0^{-1}(0)$.
Since $(g_0,g_1)\cdot (x,\theta,y) = ((g_0,g_1)\cdot x, g_0\theta g_1^{-1}, 
{\rm Ad}_{g_1}y)$, we have $\tilde{\theta} := g_0\theta g_1^{-1}\in T^k$.
For $x=(A, B, 0, 0)$,
\begin{eqnarray}
(g_0,g_1)x &=& (g_0 A g_1^{-1}, g_1 B g_0^{-1}, 0, 0)\nonumber\\
&=& (g_0 A \theta^{-1} g_0^{-1} \tilde{\theta}, 
\tilde{\theta}^{-1}g_0 \theta B g_0^{-1}, 0,0) =:(\tilde{A},\tilde{B},0,0)
 \in M_0,\nonumber
\end{eqnarray}
then we have equalities between diagonal matrices
$g_0 A \theta^{-1} g_0^{-1} = \tilde{A}\tilde{\theta}^{-1}$ and
$g_0 \theta B g_0^{-1} = \tilde{\theta} \tilde{B}$.
By comparing the eigenvalues of both sides, we can see there exist 
$\chi\in \mathcal{S}_k $ such that
$\tilde{a}_{\chi(i)}\tilde{\theta}_{\chi(i)}^{-1} = a_i\theta_i^{-1}$ and
$\tilde{b}_{\chi(i)}\tilde{\theta}_{\chi(i)} = b_i\theta_i$ for $i=1,\cdots,k$,
where $\tilde{a}_i,\tilde{b}_i,\theta_i,\tilde{\theta}_i$ are the 
$i\times i$ components of $\tilde{A}, \tilde{B}, \theta, \tilde{\theta}$,
respectively.
This implies that $(x,t,y)$ and $(g_0,g_1)\cdot (x,\theta,y)$ lie on the same
$H_0$-orbit, since $y = \tau (x)$.
\end{proof}

\begin{proof}[Proof of Theorem \ref{taubnuthilb}]
It is easy to see that $\sigma_0^{-1}(0)/H_0 \to \sigma^{-1}(0)/H$
preserves the hyper-K\"ahler structures.
Since $\sigma_0^{-1}(0)/H_0 = (\sigma_0^{-1}(0)/T^k)/\mathcal{S}_k$ and
$\sigma_0^{-1}(0)/T^k$ is isomorphic to $(\mathbb{C}^2_{Taub-NUT})^k$,
therefore $\sigma^{-1}(0)/H$ contains 
$(\mathbb{C}^2_{Taub-NUT})^k/\mathcal{S}_k$ as a hyper-K\"ahler suborbifold.
From Theorem \ref{1.1}, the quotient space $\sigma^{-1}(0)/H$ is 
homeomorphic to $\mu^{-1}(0)/H_{\rho}$, 
which is $(\mathbb{C}^2)^k/\mathcal{S}_k$ by \cite{Na1}.
Since $(\mathbb{C}^2)^k/\mathcal{S}_k$ is connected, and 
$(\mathbb{C}^2_{Taub-NUT})^k/\mathcal{S}_k$ is complete,
the embedding $\sigma_0^{-1}(0)/H_0 \to \sigma^{-1}(0)/H$ should be isomorphic.
\end{proof}

\subsection{Quiver varieties}
The setting considered in Section \ref{hilb} can be generalized to 
quiver varieties defined by Nakajima \cite{Na2}, which contains 
ALE spaces constructed by \cite{K1}.
Quiver varieties are constructed as \hK quotient as follows.

Let $Q = (V,E,s,t)$ be a finite oriented graph, that is, 
$V$ and $E$ are finite 
sets with maps $s,t:E \to V$, where $s(h)\in V$ is a source of 
a quiver $h\in E$, and $t(h)\in V$ is a target.
More over $E$ is decomposed into $E = \Omega \sqcup \overline{\Omega}$, 
with one to one correspondence $\Omega \to \overline{\Omega}$ denoted by
$h\mapsto \bar{h}$ satisfying $s(\bar{h}) = t(h)$ and $t(\bar{h}) = s(h)$.
Next we fix a dimension vector $v= (v_k)_{k\in V}$, where each $v_k$ is a 
positive integer.
Then the action of $\prod_{k\in V} U(v_k)$ on 
\begin{eqnarray}
M = \bigoplus_{h\in \Omega} {\rm Hom} (\mathbb{C}^{v_{s(h)}},
\mathbb{C}^{v_{t(h)}})
\oplus \bigoplus_{h\in \Omega} {\rm Hom} (\mathbb{C}^{v_{t(h)}}, 
\mathbb{C}^{v_{s(h)}})\label{quiverrep}
\end{eqnarray}
by $(g_k)_k\cdot (A_h,B_h)_h := (g_{t(h)}A_h g_{s(h)}^{-1},
g_{s(h)} B_h g_{t(h)}^{-1})_h$.
Then the quiver varieties are constructed by taking \hK quotients 
for this situation.

Here we explain the settings of Taub-NUT deformations for quiver varieties,
which contain the case of \cite{D}.
Let $M$ be as (\ref{quiverrep}).
We define $H,G,\rho$ as follows so that $\Hcapg = \prod_{k\in V} U(v_k)$.
We take another finite oriented graph 
$\tilde{Q} = (\tilde{V}, E, \tilde{s},\tilde{t} )$ with a surjection 
$\pi : \tilde{V} \to V$ satisfying $\pi(\tilde{s}(h)) = s(h)$ and 
$\pi(\tilde{t}(h)) = t(h)$ for all $h\in H$.
We label elements of $\pi^{-1}(k)$ numbers as
$\pi^{-1}(k) = \{ k_1,k_2,\cdots, k_{N_k} \}$.
Note that $\tilde{Q}$ may be disconnected even if $Q$ is a connected graph. 
A dimension vector $v' = (v_{\tilde{k}})_{\tilde{k}\in \tilde{V}}$
is determined by $v_{\tilde{k}} = v_{\pi{\tilde{k}}}$ 
for all $\tilde{k}\in \tilde{V}$.
Then we define $H := \prod_{\tilde{k}\in \tilde{V}} U(v_{\tilde{k}})$ and 
$G := \prod_{k\in V'} U(v_k)^{N_k - 1}$, where $V' = \{k\in V;
\ \sharp \pi^{-1}(k) \ge 2\}$. 
A homomorphism $\rho : H \to G\times G$ is defined by
\begin{eqnarray}
\rho ((g_{\tilde{k}})_{\tilde{k}\in \tilde{V}}) 
= ((g_{k_1},g_{k_2},\cdots, g_{k_{N_k - 1}}),
(g_{k_2},\cdots, g_{k_{N_k - 1}}, g_{k_{N_k}}))_{k\in V'}\nonumber
\end{eqnarray}
Then $\mu^{-1}(\iota^*\zeta)/H_{\rho}$ becomes a quiver variety, and 
we obtain another \hK quotient $\sigma^{-1}(\zeta)/H$ diffeomorphic to 
$\mu^{-1}(\iota^*\zeta)/H_{\rho}$.

\subsection{Toric hyper-K\"ahler varieties}
In the previous sections we assumed that $H$ and $G$ are compact.
However, the compactness is not essential for the proof of Theorem \ref{1.1},
we need only ${\rm Ad}_G$-invariant positive definite inner products on 
its Lie algebra and the existence of hyper-K\"ahler metrics on $N_G$ with 
tri-Hamiltonian $G\times G$-actions.
In this subsection we consider the case of noncompact abelian Lie groups
$H = \mathbb{R}^N$ and $G = \mathbb{R}^N/\mathbf{k}$, where 
the vector subspace $\mathbf{k} \subset \mathbb{R}^N$ is given by
$\mathbf{k} = \mathbf{k}_{\mathbb{Z}}\otimes \mathbb{R}$ for some 
submodule $\mathbf{k}_{\mathbb{Z}} \in \mathbb{Z}^N$.
$\rho : H \to G\times G$ is defined by $\rho(v):= 
(v\ {\rm mod}.\ \mathbf{k}, 0)$, then $\bar{\rho}$ defined by (\ref{rho1})
is surjective.
In this case we put $N_G:= G\times G\times G\times G$ with the Euclidean metric,
and $G\times G$-action on $N_G$ is defined by
$(g_0,g_1)\cdot (h_0,h_1,h_2,h_3):= (h_0 + g_0 - g_1,h_1,h_2,h_3)$.
Then Theorem \ref{Dancer}, \ref{K2Do} and Proposition \ref{kahlerpotential}
hold in this case.
Let $M = \mathbb{H}^N$, and define $H$-action on $M$ by
\begin{eqnarray}
(t_1,\cdots, t_N)\cdot (x_1,\cdots,x_N) 
:= (x_1e^{-2\pi it_1}, \cdots, x_Ne^{-2\pi it_N}).\nonumber
\end{eqnarray}
The the hyper-K\"ahler quotient $\mu^{-1}(\iota^*\zeta)/H_{\rho}$ becomes
a toric hyper-K\"ahler variety, and $\sigma^{-1}(\zeta)/H$ is its Taub-NUT
deformation defined in \cite{B}.
Theorem \ref{1.1} can be also applied to this situation.

\subsection{Hyper-K\"ahler manifolds with tri-Hamiltonian actions}
Here we show that the limited case of Theorem 7 of \cite{DS2} also 
follows from Theorem \ref{1.1}.
Let $M = \mathbb{H}^N$ and $H\subset Sp(N)$. 
Take a normal closed subgroup $\Hcapg \subset H$ and put $G := H/\Hcapg$.
Let $\rho : H \to G\times G$ be given by $\rho(h):= (1,h\Hcapg)$.
Then $X = \mu^{-1}(\iota^*\zeta)/H_{\rho}$ is a \hK manifolds with 
tri-Hamiltonian 
$G$-action, and $\sigma^{-1}(\zeta)/H$ is the {\it modification} of 
$\mu^{-1}(\iota^*\zeta)/H_{\rho}$ defined in Section 5 of \cite{DS2}.
From Theorem \ref{1.1}, we have the following results.
\begin{thm}
Let $X = \mu^{-1}(\iota^*\zeta)/H_{\rho}$ be a tri-Hamiltonian 
$G$ \hK manifold defined as above. 
Then the {\it modification} of $X$ in the sense of Section 5 of \cite{DS2} 
by the tri-Hamiltonian $G$-action is isomorphic to 
$X$ as holomorphic symplectic manifolds, 
hence diffeomorphic.
\label{modi}
\end{thm}
By Theorem 7 of \cite{DS2}, we have already known that 
$\sigma^{-1}(\zeta)/H$ is diffeomorphic to 
$\hat{\mu}^{-1}(\nu^1(N_G) + \zeta)/\Hcapg$,
which is an open subset of $X$.
Theorem \ref{modi} asserts that this open subset is 
diffeomorphic to $X$, even if it is a proper subset of $X$.
\\

{\bf Acknowledgment.} The author was supported by Grant-in-Aid for JSPS Fellows ($23\cdot 1432$)

\bibliographystyle{plain}

\end{document}